\documentclass[12pt,reqno]{amsart}
\usepackage{txfonts,fullpage}
\usepackage[dvipdfmx]{graphicx}

\newtheorem{thm}{Theorem}[section]

\newtheorem{prop}[thm]{Proposition}
\newtheorem{lem}[thm]{Lemma}

\newtheorem{rem}[thm]{Remark}
\numberwithin{equation}{section}
\allowdisplaybreaks

\begin{document}

\title{Generalized $q$-Painlev\'{e} VI systems of type $(A_{2n+1}+A_1+A_1)^{(1)}$ arising from cluster algebra}
\date{}
\author{Naoto Okubo}
\address{Department of Physics and Mathematics, Aoyama Gakuin University, 5-10-1, Fuchinobe, Chuo-ku, Sagamihara-shi, Kanagawa 252-5258, Japan}
\email{okubo@gem.aoyama.ac.jp}
\author{Takao Suzuki}
\address{Department of Mathematics, Kindai University, 3-4-1, Kowakae, Higashi-Osaka, Osaka 577-8502, Japan}
\email{suzuki@math.kindai.ac.jp}

\maketitle

\begin{abstract}
In this article we formulate a group of birational transformations which is isomorphic to an extended affine Weyl group of type $(A_{2n+1}+A_1+A_1)^{(1)}$ with the aid of mutations and permutations of vertices to a mutation-periodic quiver on a torus.
This group provides a class of higher order generalizations of Jimbo-Sakai's $q$-Painlev\'e VI equation as translations on a root lattice.
Then the known three systems are obtained again; the $q$-Garnier system, a similarity reduction of the lattice $q$-UC hierarchy and a similarity reduction of the $q$-Drinfeld-Sokolov hierarchy.


\end{abstract}

\section{Introduction}\label{Sec:Intro}

The Painlev\'e equations were first discovered at the beginning of twentieth century.
Painlev\'e and Gambier tried to classify second order meromorphic ordinary differential equations without movable branch points.
This property of a differential equation is now called the Painlev\'e property.
As a result, six types of nonlinear differential equations were obtained.
At around the same time, these equations were given by the deformation theory of second order linear ordinary differential equations.

Nearly a century later, Grammaticos, Ramani and Papageorgiou proposed a discrete analogue of the Painlev\'e property called singularity confinement in \cite{GRP}.
That was a trigger for the discovery of various discrete Painlev\'e equations.
Then it became the next problem to reveal how many second order discrete Painlev\'e equations exist.
An answer to this problem was given by Sakai in \cite{Sak1}.
According to that, the second order discrete Painlev\'e equations are classified by the geometry of rational surfaces called the initial value spaces, which are characterized by pairs of affine root systems.
We list them in the following table.
\[\renewcommand{\arraystretch}{1.2}\begin{array}{|c|ccccccc|}\hline
	\text{Difference type} & \multicolumn{7}{l|}{\text{Symmetry/Surface type}} \\\hline
	\text{elliptic} & E_8/A_0 & & & & & & \\
	\text{multiplicative} & E_8/A_0 & E_7/A_1 & E_6/A_2 & D_5/A_3 & A_4/A_4 & A_2+A_1/A_5 & A_1+{A_1\atop{|\alpha|^2=14}}/A_6 \\
	& {A_1\atop{|\alpha|^2=8}}/A_7 & A_1/A_7 & A_0/A_8 & & & & \\
	\text{additive} & E_8/A_0 & E_7/A_1 & E_6/A_2 & D_4/D_4 & A_3/D_5 & A_1+A_1/D_6 & A_2/E_6 \\
	& {A_1\atop{|\alpha|^2=4}}/D_7 & A_1/E_7 & A_0/D_8 & A_0/E_8 & & & \\\hline
\end{array}\]
Note that all of the continuous Painlev\'e equations can be regarded as continuous limits of some multiplicative ($q$-difference) Painlev\'e equations or continuous flows commuting with some additive Painlev\'e equations.

The cluster algebra was introduced by Fomin and Zelevinsky in \cite{FZ1,FZ2}.
Let $Q$ be a quiver with $N$ vertices.
We assume that a quiver doesn't have any loop or any 2-cycle.
Also let ${\boldsymbol x}=(x_1,\ldots,x_N)$ be an $N$-tuple of cluster variables and ${\boldsymbol y}=(y_1,\ldots,y_N)$ an $N$-tuple of coefficients.
We call the triple $(Q,{\boldsymbol x},{\boldsymbol y})$ a seed.
Let $\Lambda=\left(\lambda_{i,j}\right)_{i,j=1}^N$ be a skew-symmetric matrix corresponding to the quiver $Q$.
In other words, if there are $l$ arrows from $i$ to $j$, then we set $\lambda_{i,j}=l$ and $\lambda_{j,i}=-l$.
For each $k\in\{1,\ldots,N\}$, we define a mutation $\mu_k:(Q,{\boldsymbol x},{\boldsymbol y})\to(Q',{\boldsymbol x}',{\boldsymbol y}')$ by
\begin{align*}
	\lambda_{i,j}' &= \left\{\begin{array}{ll}
		-\lambda_{i,j} & (i=k\vee j=k) \\[4pt]
		\lambda_{i,j}+\lambda_{i,k}\,\lambda_{k,j} & (\lambda_{i,k}>0\wedge\lambda_{k,j}>0) \\[4pt]
		\lambda_{i,j}-\lambda_{i,k}\,\lambda_{k,j} & (\lambda_{i,k}<0\wedge\lambda_{k,j}<0) \\[4pt]
		\lambda_{i,j} & (\text{otherwise})
	\end{array}\right., \\
	x_i' &= \left\{\begin{array}{ll}
		\displaystyle\frac{1}{1+y_k}\frac{\prod_{\lambda_{k,j}>0}x_j^{\lambda_{k,j}}}{x_k}+\frac{1}{1+y_k^{-1}}\frac{\prod_{\lambda_{k,j}<0}x_j^{-\lambda_{k,j}}}{x_k} & (i=k) \\[4pt]
		x_i & (i\neq k)
	\end{array}\right., \\
	y_i' &= \left\{\begin{array}{ll}
		y_k^{-1} & (i=k) \\[4pt]
		y_i\,(1+y_k^{-1})^{-\lambda_{k,i}} & (\lambda_{k,i}\geq0) \\[4pt]
		y_i\,(1+y_k)^{-\lambda_{k,i}} & (\lambda_{k,i}<0) \\[4pt]
		y_i & (\text{otherwise})
	\end{array}\right..
\end{align*}
The cluster algebra with coefficients is defined as a variety of commutative ring $\mathbb{Z}({\boldsymbol y}_0)[{\boldsymbol x}|{\boldsymbol x}\in X]$, where $X$ is a set of all cluster variables given by iterative mutations to an initial seed $(Q_0,{\boldsymbol x}_0,{\boldsymbol y}_0)$.

The property of a quiver called mutation-period is that iterative mutations give a permutation of vertices of the quiver.
It was introduced by Nakanishi in \cite{Nak}.
On the other hand, as is seen above, new cluster variables $x_i'$ (resp. coefficients $y_i'$) are rational in original cluster variables and coefficients $x_i,y_i$ (resp. coefficients $y_i$).
Thanks to these two properties, some mutation-periodic quivers become sources of discrete integrable systems (\cite{IIKKN1,IIKKN2,IIKNS,IN,Nob,O1}) or $q$-Painlev\'e equations (\cite{BGM,HI,O2}).

Higher order generalizations of the $q$-Painlev\'e equations have been proposed from some points of view; birational representations of affine Weyl groups (\cite{KNY1,KNY2,M,TT}), a cluster mutation (\cite{HI}), a $q$-analogue of the isomonodromy deformation (\cite{Sak2}), similarity reductions of discrete integrable systems (\cite{Suz2,Suz3,T3}) and a Pad\'e method (\cite{NagY1,NagY2}).
However there doesn't exist any theory which governs all of them unlike in the case of second order.
Our purpose is to establish a good classification theory based on the affine root systems and the cluster algebra.

\begin{figure}
	\begin{center}
	\begin{picture}(90,90)
		\put(0,0){\small$8$}\put(2,11){\vector(0,1){66}}\put(6,11){\vector(1,4){12}}
		\put(84,0){\small$3$}\put(78,2){\vector(-1,0){66}}\put(78,7){\vector(-4,1){48}}
		\put(18,18){\small$5$}\put(18,29){\vector(-1,4){12}}\put(22,29){\vector(0,1){30}}
		\put(66,18){\small$2$}\put(60,23){\vector(-1,0){30}}\put(60,18){\vector(-4,-1){48}}
		\put(18,64){\small$1$}\put(30,66){\vector(1,0){30}}\put(30,70){\vector(4,1){48}}
		\put(66,64){\small$6$}\put(68,59){\vector(0,-1){30}}\put(72,59){\vector(1,-4){12}}
		\put(0,82){\small$4$}\put(12,82){\vector(4,-1){48}}\put(12,86){\vector(1,0){66}}
		\put(84,82){\small$7$}\put(84,77){\vector(-1,-4){12}}\put(88,77){\vector(0,-1){66}}
	\end{picture}
	\caption{$q$-$P_{\rm VI}$ quiver}\label{Fig:q-PVI}
	\end{center}
\end{figure}
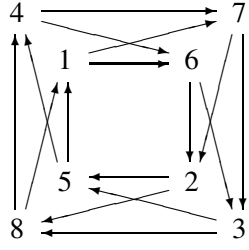

As a first step, we focus on the $q$-Painlev\'e equation of type $D_5/A_3$ which is also known as the $q$-Painlev\'e VI equation.
The $q$-Painlev\'e VI equation was first proposed by Jimbo and Sakai as a $q$-analogue of the isomonodromy deformation of the second order Fuchsian differential equation with four regular singular points in \cite{JS}.
Afterward, it was given as a birational representation of an extended affine Weyl group of type $D_5^{(1)}$ in \cite{Sak1,TM}.
Since the $q$-Painlev\'e VI equation and the corresponding affine Weyl group have been already derived from the mutation-periodic quiver in Figure \ref{Fig:q-PVI} in \cite{BGM,O2}, our first aim is to extend these previous works.

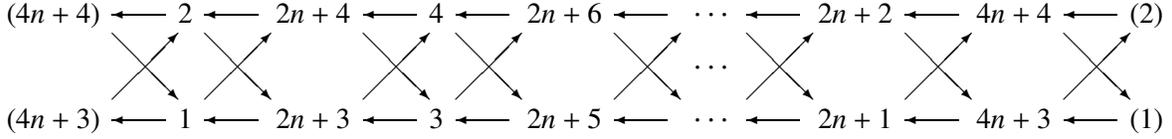
\begin{figure}
	\begin{center}
	\begin{picture}(440,50)
		\put(0,0){\small$(4n+3)$}\put(0,40){\small$(4n+4)$}
		\put(40,11){\vector(1,1){25}}\put(40,36){\vector(1,-1){25}}\put(60,3){\vector(-1,0){20}}\put(60,43){\vector(-1,0){20}}
		\put(65,0){\small$1$}\put(65,40){\small$2$}
		\put(75,11){\vector(1,1){25}}\put(75,36){\vector(1,-1){25}}\put(95,3){\vector(-1,0){20}}\put(95,43){\vector(-1,0){20}}
		\put(102,0){\small$2n+3$}\put(102,40){\small$2n+4$}
		\put(135,11){\vector(1,1){25}}\put(135,36){\vector(1,-1){25}}\put(155,3){\vector(-1,0){20}}\put(155,43){\vector(-1,0){20}}
		\put(160,0){\small$3$}\put(160,40){\small$4$}
		\put(170,11){\vector(1,1){25}}\put(170,36){\vector(1,-1){25}}\put(190,3){\vector(-1,0){20}}\put(190,43){\vector(-1,0){20}}
		\put(197,0){\small$2n+5$}\put(197,40){\small$2n+6$}
		\put(230,11){\vector(1,1){25}}\put(230,36){\vector(1,-1){25}}\put(250,3){\vector(-1,0){20}}\put(250,43){\vector(-1,0){20}}
		\put(260,0){$\cdots$}	\put(260,20){$\cdots$}\put(260,40){$\cdots$}
		\put(280,11){\vector(1,1){25}}\put(280,36){\vector(1,-1){25}}\put(300,3){\vector(-1,0){20}}\put(300,43){\vector(-1,0){20}}
		\put(307,0){\small$2n+1$}\put(307,40){\small$2n+2$}
		\put(340,11){\vector(1,1){25}}\put(340,36){\vector(1,-1){25}}\put(360,3){\vector(-1,0){20}}\put(360,43){\vector(-1,0){20}}
		\put(367,0){\small$4n+3$}\put(367,40){\small$4n+4$}
		\put(400,11){\vector(1,1){25}}\put(400,36){\vector(1,-1){25}}\put(420,3){\vector(-1,0){20}}\put(420,43){\vector(-1,0){20}}
		\put(425,0){\small$(1)$}\put(425,40){\small$(2)$}
	\end{picture}
	\caption{Generalized $q$-$P_{\rm VI}$ quiver}\label{Fig:Gen_q-PVI}
	\end{center}
\end{figure}

In this article we consider a mutation-periodic quiver in Figure \ref{Fig:Gen_q-PVI} which is a natural extension of that in Figure \ref{Fig:q-PVI}.
Then we find some compositions of iterative mutations and a permutation of vertices such that the quiver is invariant under their actions.
These compositions turn out to be generators of a group of birational transformations which is isomorphic to an extended affine Weyl group of type $(A_{2n+1}+A_1+A_1)^{(1)}$.
We also give its abelian normal subgroup generated by translations on a root lattice.
This group of translations provides a class of generalized $q$-Painlev\'e VI systems containing the known three systems; Sakai's $q$-Garnier system (\cite{Sak2}), Tsuda's $q$-Painlev\'e system arising from the $q$-LUC hierarchy (\cite{T3}) and the $q$-Painlev\'e system $q$-$P_{(n+1,n+1)}$ arising from the $q$-DS hierarchy (\cite{Suz2,Suz3}).

\begin{rem}\label{Rem:MOT}
In a recent work \cite{MOT} a birational representation of an affine Weyl group of type $(A_{LM-1}+A_{LN-1}+A_{L-1})^{(1)}$, where $M$ and $N$ are coprime, is derived from a mutation-periodic quiver on a torus.
It contains the result of \cite{KNY1} as the case $L=1$.
Moreover, in the case $(L,M,N)=(2,1,n+1)$, the quiver coincides with that in Figure \ref{Fig:Gen_q-PVI}.
\end{rem}

\begin{rem}
It is shown in \cite{NagY2} that both the $q$-Garnier system and $q$-$P_{(n+1,n+1)}$ are derived from the same linear $q$-difference equation.
There is a difference in directions of discrete time evolutions between them.
Besides, both Tsuda's $q$-Painlev\'e system and $q$-$P_{(n+1,n+1)}$ reduce in a continuous limit $q\to1$ to the same Hamiltonian system given in \cite{FS1,Suz1,T4}, which is a representative isomonodromy deformation equation as well as the continuous Garnier system is.
The result of this article clarifies the connection between those systems in a framework of the affine root systems and the cluster algebra.
\end{rem}

\begin{rem}
The system given in this article is actually a deautonomization of the dimer (or Goncharov-Kenyon) integrable system defined by the Newton polygon which is the right triangle with catheti of integer lengths $2$ and $2n$.
Three factors in the reducible affine Weyl group correspond to the integer points on the sides of the triangle, that is the affine Weyl group of type $A_{N-1}^{(1)}$ arises from $N-1$ integer points on one side.
For its detail, see \cite{BGM,GK,ILP1,ILP2}.
\end{rem}

\begin{rem}
Masuda proposed another generalized $q$-Painlev\'e VI system called the $q$-Sasano system as a birational representation of an extended affine Weyl group of type $D^{(1)}_{2n+5}$ in \cite{M}.
A relationship between the $q$-Sasano system and the cluster algebra is investigated in \cite{MOT}.
\end{rem}

This article is organized as follows.
In Section \ref{Sec:Aff_Wey}, we recall the definition of the affine Weyl group of type $A_{N-1}^{(1)}$.
In Section \ref{Sec:Bir_Rep}, we formulate a birational representation of an extended affine Weyl group of type $(A_{2n+1}+A_1+A_1)^{(1)}$.
In Section \ref{Sec:Gen_q-PVI}, we give a group of translations which provides a class of generalized $q$-Painlev\'e VI systems.
In Section \ref{Sec:q-FST}, \ref{Sec:q-Garnier} and \ref{Sec:q-Tsuda}, we derive $q$-$P_{(n+1,n+1)}$, the $q$-Garnier system and Tsuda's $q$-Painlev\'e system respectively.
Since some formulas or theorems are shown by large calculations, we place them in appendices.

\section{Affine Weyl group of type $A_{N-1}^{(1)}$}\label{Sec:Aff_Wey}

In this article we denote the quotient ring $\mathbb{Z}/N\mathbb{Z}$ by $\mathbb{Z}_N$ for the sake of simplicity.
Let $A=\left(a_{i,j}\right)_{i,j\in\mathbb{Z}_N}$ be a Cartan matrix defined by
\[
	a_{0,0} = a_{1,1} = 2,\quad
	a_{0,1} = a_{1,0} = -2,
\]
for $N=2$ and
\[
	a_{i,j} = \left\{\begin{array}{ll}
		2 & (i=j) \\[4pt]
		-1 & (i\equiv j\pm1\bmod N) \\[4pt]
		0 & (\text{otherwise})
	\end{array}\right.,
\]
for $N\geq3$.
Then a Weyl group $W(A)$ is the affine Weyl group of type $A_{N-1}^{(1)}$ with the generators $r_i$ $(i\in\mathbb{Z}_N)$ and the fundamental relations
\[
	r_0^2 = r_1^2 = 1,
\]
for $N=2$ and
\begin{align*}
	&r_i^2 = 1\quad (i\in\mathbb{Z}_N), \\
	&r_i\,r_j\,r_i = r_j\,r_i\,r_j\quad (i,j\in\mathbb{Z}_N,\ a_{i,j}=-1), \\
	&r_i\,r_j = r_j\,r_i\quad (i,j\in\mathbb{Z}_N,\ a_{i,j}=0),
\end{align*}
for $N\geq3$.
We denote the affine Weyl group of type $A_{N-1}^{(1)}$ by $W(A_{N-1}^{(1)})$.

The group $W(A_{N-1}^{(1)})$ can be realized as that generated by simple reflections acting on a root lattice.
Let $\alpha_i$ $(i\in\mathbb{Z}_N)$ be multiplicative simple roots and $q$ a multiplicative null root.
We assume that
\[
	\alpha_0\ldots\alpha_{N-1} = q.
\]
Also let $r_i$ $(i\in\mathbb{Z}_N)$ be simple reflections acting on the simple roots as
\[
	r_j(\alpha_i) = \alpha_i\,\alpha_j^{-a_{j,i}}\quad (i,j\in\mathbb{Z}_N).
\]
Then the group $\langle r_0,\ldots,r_{N-1}\rangle$ is isomorphic to $W(A_{N-1}^{(1)})$, that is the simple reflections satisfy the fundamental relations.

Let $\pi$ be a cyclic permutation acting on the natural numbers defined by
\[
	\pi = (0,1,2,\ldots,N-2,N-1) = (N-2,N-1)\ldots(1,2)(0,1),
\]
which preserves the Cartan matrix $A$ as $a_{\pi(i),\pi(j)}=a_{i,j}$.
It is called a Dynkin diagram automorphism and can be lifted to an action on the simple roots as
\[
	\pi(\alpha_i) = \alpha_{i+1}\quad (i\in\mathbb{Z}_N).
\]
Then we obtain an extended affine Weyl group $W(A_{N-1}^{(1)})\rtimes\langle\pi\rangle$ with fundamental relations
\[
	\pi^{N+1} = 1,\quad
	\pi\,r_i = r_{i+1}\,\pi\quad (i\in\mathbb{Z}_N).
\]

Let $T_i$ $(i\in\mathbb{Z}_N)$ be transformations defined by
\begin{equation}\label{Eq:Translation}
	T_0 = r_1\ldots r_{N-1}\,\pi,\quad
	T_i = r_{i+1}\ldots r_{N-1}\,\pi\,r_1\ldots r_i\quad (i=1,\ldots,N-2),\quad
	T_{N-1} = \pi\,r_1\ldots r_{N-1}.
\end{equation}
They are called translations and act on the simple roots as
\[
	T_j(\alpha_i) = q^{\delta_{i,j}-\delta_{i,j+1}}\alpha_i\quad (i,j\in\mathbb{Z}_N),
\]
where the symbol $\delta_{i,j}$ stands for the Kronecker's delta.
Note that
\[
	T_0\ldots T_{N-1} = 1.
\]
Then the translations generate an abelian normal subgroup of $W(A_{N-1}^{(1)})\rtimes\langle\pi\rangle$ with fundamental relations
\[
	r_i\,T_j = T_{(i-1,i)(j)}\,r_i,\quad
	\pi\,T_i = T_{i-1}\,\pi\quad (i,j\in\mathbb{Z}_N),
\]
where the symbol $(i-1,i)(j)$ stands for the action of the permutation $(i-1,i)$ on the number $j$.
Note that the extended affine Weyl group is decomposed into semi-direct product of the group of the translations and a finite Weyl group as
\[
	W(A_{N-1}^{(1)})\rtimes\langle\pi\rangle = \langle T_0,\ldots,T_{N-1}\rangle\rtimes\langle r_1,\ldots,r_{N-1}\rangle.
\]

\section{Birational representation of affine Weyl group of type $(A_{2n+1}+A_1+A_1)^{(1)}$}\label{Sec:Bir_Rep}

Let $Q$ be a quiver given in Figure \ref{Fig:Gen_q-PVI} with $n\geq1$.
Also let $(Q,(x_1,\ldots,x_{4n+4}),(y_1,\ldots,y_{4n+4}))$ be a seed.
Then we can describe the actions of the mutations $\mu_1,\ldots,\mu_{4n+4}$ on the coefficients $y_1,\ldots,y_{4n+4}$ following the definition given in Section \ref{Sec:Intro}.
We don't give their explicit formulas here.
Moreover, we consider permutations of vertices of the quiver $(i,j)$ acting on the coefficients as
\[
	(y_1,\ldots,y_{4n+4}) \xrightarrow{(i,j)} (y_1,\ldots,y_{i-1},y_j,y_{i+1},\ldots,y_{j-1},y_i,y_{j+1},\ldots,y_{4n+4}),
\]
for $i,j=1,\ldots,4n+4$ with $i<j$.
Note that we don't consider the cluster variables.
In this section we give some compositions of iterative mutations and permutations of vertices such that the quiver $Q$ is invariant under their actions.
These compositions turn out to be simple reflections or Dynkin diagram automorphisms and generate a group of birational transformations which is isomorphic to an extended affine Weyl group of type $(A_{2n+1}+A_1+A_1)^{(1)}$.

Let $\alpha_i$ $(i\in\mathbb{Z}_{2n+2})$, $\beta_k$ $(k\in\mathbb{Z}_2)$ and $\beta'_k$ $(k\in\mathbb{Z}_2)$ be parameters corresponding to the multiplicative simple roots for $W(A^{(1)}_{2n+1})$, $W(A^{(1)}_1)$ and $W(A^{(1)}_1)$ respectively defined by
\begin{align*}
	&\alpha_{2i} = y_{2i+1}\,y_{2i+2},\quad
	\alpha_{2i+1} = y_{2i+2n+3}\,y_{2i+2n+4}\quad (i=0,\ldots,n), \\
	&\beta_0 = \prod_{i=0}^{n}y_{2i+1}\,y_{2i+2n+4},\quad
	\beta_1 = \prod_{i=0}^{n}y_{2i+2}\,y_{2i+2n+3},\quad
	\beta'_0 = \prod_{i=0}^{n}y_{2i+1}\,y_{2i+2n+3},\quad
	\beta'_1 = \prod_{i=0}^{n}y_{2i+2}\,y_{2i+2n+4}.
\end{align*}
Also let $\varphi_i$ $(i\in\mathbb{Z}_{2n+2})$ be dependent variables defined by
\[	\varphi_{2i} = y_{2i+1},\quad
	\varphi_{2i+1} = y_{2i+2n+3}\quad (i=0,\ldots,n).
\]
Since a product of all coefficients $\prod_{i=1}^{4n+4}y_i$ is invariant under the action of any mutation or permutation of vertices, we denote it by $q$.
Then we obtain
\[
	\prod_{i=0}^{2n+1}\alpha_i = \beta_0\,\beta_1 = \beta'_0\,\beta'_1 = q,\quad
	\beta_0 = \prod_{i=0}^{n}\frac{\varphi_{2i}\,\alpha_{2i+1}}{\varphi_{2i+1}},\quad
	\beta'_0 = \prod_{i=0}^{n}\varphi_{2i}\,\varphi_{2i+1}.
\]
In the following, we denote the Cartan matrix of type $W(A^{(1)}_{2n+1})$ and $W(A^{(1)}_1)$ by $\left(a_{i,j}\right)_{i,j}$ and $\left(b_{k,l}\right)_{k,l}$ respectively.

Simple reflections $r_i$ $(i\in\mathbb{Z}_{2n+2})$ corresponding to the parameters $\alpha_i$ are defined by
\[
	r_{2i} = \mu_{2i+1}\,(2i+1,2i+2)\,\mu_{2i+1},\quad
	r_{2i+1} = \mu_{2i+2n+3}\,(2i+2n+3,2i+2n+4)\,\mu_{2i+2n+3},
\]
for $i=0,\ldots,n$.
We can show easily that the quiver $Q$ is invariant under their actions.
They act on the parameters as
\[
	r_j(\alpha_i) = \alpha_i\,\alpha_j^{-a_{j,i}},\quad
	r_j(\beta_k) = \beta_k,\quad
	r_j(\beta'_k) = \beta'_k\quad (i,j\in\mathbb{Z}_{2n+2},\ k\in\mathbb{Z}_2),
\]
and the dependent variables as
\[
	r_j(\varphi_i) = \varphi_i\,\alpha^{-\delta_{i,j}+\delta_{i,j+1}}\left(\frac{\alpha_j+\varphi_j}{1+\varphi_j}\right)^{\delta_{i,j-1}-\delta_{i,j+1}}\quad (i,j\in\mathbb{Z}_{2n+2}).
\]
These formulas are derived by direct calculations.
For example, the simple reflection $r_0$ acts on the coefficients as
\begin{align*}
	&(y_1,\ldots,y_{4n+4}) \\
	&\xrightarrow{\mu_1} \left(\frac{1}{y_1},y_2,y_3,\ldots,y_{2n+2},(1+y_1)\,y_{2n+3},\frac{y_{2n+4}}{1+\frac{1}{y_1}},y_{2n+5},\ldots,y_{4n+2},\frac{y_{4n+3}}{1+\frac{1}{y_1}},(1+y_1)\,y_{4n+4}\right) \\
	&\xrightarrow{(1,2)} \left(y_2,\frac{1}{y_1},y_3,\ldots,y_{2n+2},(1+y_1)\,y_{2n+3},\frac{y_{2n+4}}{1+\frac{1}{y_1}},y_{2n+5},\ldots,y_{4n+2},\frac{y_{4n+3}}{1+\frac{1}{y_1}},(1+y_1)\,y_{4n+4}\right) \\
	&\xrightarrow{\mu_1} \left(\frac{1}{y_2},\frac{1}{y_1},y_3,\ldots,y_{2n+2},\frac{1+y_1}{1+\frac{1}{y_2}}y_{2n+3},\frac{1+y_2}{1+\frac{1}{y_1}}y_{2n+4},y_{2n+5},\ldots,y_{4n+2},\frac{1+y_2}{1+\frac{1}{y_1}}y_{4n+3},\frac{1+y_1}{1+\frac{1}{y_2}}y_{4n+4}\right),
\end{align*}
from which we obtain
\begin{align*}
	&(\alpha_0,\ldots,\alpha_{2n+1},\beta_0,\beta_1,\beta'_0,\beta'_1,\varphi_0,\ldots,\varphi_{2n+1}) \\
	&\xrightarrow{r_0} \left(\frac{1}{\alpha_0},\alpha_0\,\alpha_1,\alpha_2,\ldots,\alpha_{2n},\alpha_0\,\alpha_{2n+1},\beta_0,\beta_1,\beta'_0,\beta'_1,\frac{\varphi_0}{\alpha_0},\varphi_1\frac{\alpha_0\,(1+\varphi_0)}{\alpha_0+\varphi_0},\varphi_2,\ldots,\varphi_{2n},\varphi_{2n+1}\frac{\alpha_0+\varphi_0}{1+\varphi_0}\right).
\end{align*}
Simple reflections $s_k,s'_k$ $(k\in\mathbb{Z}_2)$ corresponding to the parameters $\beta_k,\beta'_k$ are defined by
\begin{align*}
	s_0 &= \mu_1\,\mu_{2n+4}\,\mu_3\,\mu_{2n+6}\ldots\mu_{2n-1}\,\mu_{4n+2}\,\mu_{2n+1}\,(2n+1,4n+4)\,\mu_{2n+1}\,\mu_{4n+2}\,\mu_{2n-1}\ldots\mu_{2n+6}\,\mu_3\,\mu_{2n+4}\,\mu_1, \\
	s_1 &= \mu_2\,\mu_{2n+3}\,\mu_4\,\mu_{2n+5}\ldots\mu_{2n}\,\mu_{4n+1}\,\mu_{2n+2}\,(2n+2,4n+3)\,\mu_{2n+2}\,\mu_{4n+1}\,\mu_{2n}\ldots\mu_{2n+5}\,\mu_4\,\mu_{2n+3}\,\mu_2, \\
	s'_0 &= \mu_1\,\mu_{2n+3}\,\mu_3\,\mu_{2n+5}\ldots\mu_{2n-1}\,\mu_{4n+1}\,\mu_{2n+1}\,(2n+1,4n+3)\,\mu_{2n+1}\,\mu_{4n+1}\,\mu_{2n-1}\ldots\mu_{2n+5}\,\mu_3\,\mu_{2n+3}\,\mu_1, \\
	s'_1 &= \mu_2\,\mu_{2n+4}\,\mu_4\,\mu_{2n+6}\ldots\mu_{2n}\,\mu_{4n+2}\,\mu_{2n+2}\,(2n+2,4n+4)\,\mu_{2n+2}\,\mu_{4n+2}\,\mu_{2n}\ldots\mu_{2n+6}\,\mu_4\,\mu_{2n+4}\,\mu_2.
\end{align*}
We show that the quiver $Q$ is invariant under their actions in Appendix \ref{App:Birat_Transf}.
They act on the parameters as
\begin{equation}\label{Eq:s0s1_Transf_1}
	s_l(\alpha_i) = \alpha_i,\quad
	s_l(\beta_k) = \beta_k\,\beta_l^{-b_{l,k}},\quad
	s_l(\beta'_k) = \beta'_k\quad (i\in\mathbb{Z}_{2n+2},\ k,l\in\mathbb{Z}_2),
\end{equation}
and
\[
	s'_l(\alpha_i) = \alpha_i,\quad
	s'_l(\beta_k) = \beta_k,\quad
	s'_l(\beta'_k) = \beta'_k\,(\beta'_l)^{-b_{l,k}}\quad (i\in\mathbb{Z}_{2n+2},\ k,l\in\mathbb{Z}_2).
\]
Their actions on the dependent variables are given by
\begin{equation}\begin{split}\label{Eq:s0s1_Transf_2}
	s_0(\varphi_{2i}) &= \frac{\varphi_{2i+1}}{\alpha_{2i+1}}\frac{\sum_{j=0}^{n}\left(\prod_{k=0}^{j-1}\varphi_{2i+2k}\,\frac{\alpha_{2i+2k+1}}{\varphi_{2i+2k+1}}\right)\left(1+\varphi_{2i+2j}\right)}{\sum_{j=0}^{n}\left(\prod_{k=0}^{j-1}\varphi_{2i+2k+2}\,\frac{\alpha_{2i+2k+3}}{\varphi_{2i+2k+3}}\right)\left(1+\varphi_{2i+2j+2}\right)}, \\
	s_0(\varphi_{2i+1}) &= \alpha_{2i+1}\,\varphi_{2i+2}\,\frac{\sum_{j=0}^{n}\left(\prod_{k=0}^{j-1}\frac{\alpha_{2i+2k+3}}{\varphi_{2i+2k+3}}\,\varphi_{2i+2k+4}\right)\left(1+\frac{\alpha_{2i+2j+3}}{\varphi_{2i+2j+3}}\right)}{\sum_{j=0}^{n}\left(\prod_{k=0}^{j-1}\frac{\alpha_{2i+2k+1}}{\varphi_{2i+2k+1}}\,\varphi_{2i+2k+2}\right)\left(1+\frac{\alpha_{2i+2j+1}}{\varphi_{2i+2j+1}}\right)}, \\
	s_1(\varphi_{2i}) &= \alpha_{2i}\,\varphi_{2i+1}\,\frac{\sum_{j=0}^{n}\left(\prod_{k=0}^{j-1}\frac{\alpha_{2i+2k+2}}{\varphi_{2i+2k+2}}\,\varphi_{2i+2k+3}\right)\left(1+\frac{\alpha_{2i+2j+2}}{\varphi_{2i+2j+2}}\right)}{\sum_{j=0}^{n}\left(\prod_{k=0}^{j-1}\frac{\alpha_{2i+2k}}{\varphi_{2i+2k}}\,\varphi_{2i+2k+1}\right)\left(1+\frac{\alpha_{2i+2j}}{\varphi_{2i+2j}}\right)}, \\
	s_1(\varphi_{2i+1}) &= \frac{\varphi_{2i+2}}{\alpha_{2i+2}}\frac{\sum_{j=0}^{n}\left(\prod_{k=0}^{j-1}\varphi_{2i+2k+1}\,\frac{\alpha_{2i+2k+2}}{\varphi_{2i+2k+2}}\right)\left(1+\varphi_{2i+2j+1}\right)}{\sum_{j=0}^{n}\left(\prod_{k=0}^{j-1}\varphi_{2i+2k+3}\,\frac{\alpha_{2i+2k+4}}{\varphi_{2i+2k+4}}\right)\left(1+\varphi_{2i+2j+3}\right)},
\end{split}\end{equation}
for $i=0,\ldots,n$ and
\[
	s'_0(\varphi_i) = \frac{1}{\varphi_{i+1}}\frac{\sum_{j=0}^{2n+1}\prod_{k=0}^{j-1}\frac{1}{\varphi_{i+k+2}}}{\sum_{j=0}^{2n+1}\prod_{k=0}^{j-1}\frac{1}{\varphi_{i+k}}},\quad
	s'_1(\varphi_i) = \frac{\alpha_i}{\frac{\varphi_{i+1}}{\alpha_{i+1}}}\frac{\sum_{j=0}^{2n+1}\prod_{k=0}^{j-1}\frac{\varphi_{i+k}}{\alpha_{i+k}}}{\sum_{j=0}^{2n+1}\prod_{k=0}^{j-1}\frac{\varphi_{i+k+2}}{\alpha_{i+k+2}}}\quad (i\in\mathbb{Z}_{2n+2}).
\]
We derive these formulas in Appendix \ref{App:Birat_Transf}.

\begin{thm}\label{Thm:Fund_Rel}
If we set
\[
	G = \langle r_0,\ldots,r_{2n+1}\rangle,\quad
	H = \langle s_0,s_1\rangle,\quad
	H' = \langle s'_0,s'_1\rangle,
\]
then the groups $G$, $H$ and $H'$ are isomorphic to the affine Weyl groups $W(A^{(1)}_{2n+1})$, $W(A^{(1)}_1)$ and $W(A^{(1)}_1)$ respectively.
Moreover, any two groups are mutually commutative, that is
\[
	GH = HG,\quad
	GH' = H'G,\quad
	HH' = H'H.
\]
\end{thm}

We prove this theorem in Appendix \ref{App:Prf_Fund_Rel}.

\begin{rem}
The definition of the simple reflections $r_i$ is suggested by \cite{BGM}, in which birational representations of extended affine Weyl groups providing the $q$-Painlev\'e equations are derived from mutation-periodic quivers.
Besides, the simple reflections $s_k,s'_k$ appeared in \cite{ILP2}.
\end{rem}

\begin{rem}
As is seen in Remark \ref{Rem:MOT}, Theorem \ref{Thm:Fund_Rel} is obtained independently in \cite{MOT}.
The proof in \cite{MOT} is given with the aid of a mutation combinatorics and hence can be applied to many other mutation-periodic quivers.
Nevertheless, we prove the theorem in Appendix \ref{App:Prf_Fund_Rel} by a direct calculation on purpose.
The reason is that some equations given throughout our proof are used in Appendix \ref{App:Prf_q-FST} to derive $q$-$P_{(n+1,n+1)}$.
Besides, the first half of the theorem is also shown in \cite{IIO,ILP2}.
\end{rem}

\begin{rem}
Only in the case $n=1$, the quiver $Q$ is invariant under the actions of permutations $(1,4)$, $(2,3)$, $(5,8)$ and $(6,7)$.
Then the group $\langle r_0,r_1,(1,4),(2,3),(5,8),(6,7)\rangle$ is isomorphic to the affine Weyl group of type $D^{(1)}_5$.
For its detail, see \cite{BGM}.
\end{rem}

In the last, we define Dynkin diagram automorphisms $\pi,\pi'$ by
\begin{align*}
	\pi &= (1,2n+3,3,2n+5,\ldots,2n+1,4n+3)(2,2n+4,4,2n+6,\ldots,2n+2,4n+4), \\
	\pi' &= (1,2n+4,3,2n+6,\ldots,2n+1,4n+4)(2,2n+3,4,2n+5,\ldots,2n+2,4n+3),
\end{align*}
and $\rho$ by
\begin{align*}
	\rho &= (1,2)(3,2n+2)(4,2n+1)\ldots(n+2,n+3) \\
	&\quad \times (2n+3,4n+3)(2n+5,4n+1)\ldots(3n+2,3n+4) \\
	&\quad \times (2n+4,4n+4)(2n+6,4n+2)\ldots(3n+3,3n+5),
\end{align*}
for $n$ is odd and
\begin{align*}
	\rho &= (1,2)(3,2n+2)(4,2n+1)\ldots(n+2,n+3) \\
	&\quad \times (2n+3,4n+3)(2n+5,4n+1)\ldots(3n+1,3n+5) \\
	&\quad \times (2n+4,4n+4)(2n+6,4n+2)\ldots(3n+2,3n+6),
\end{align*}
for $n$ is even.
We can show easily that the quiver $Q$ is invariant under their actions.
They act on the parameters as
\begin{align*}
	&\pi(\alpha_i) = \alpha_{i+1}\quad (i\in\mathbb{Z}_{2n+2}),\quad
	\pi(\beta_k) = \beta_{k+1},\quad
	\pi(\beta'_k) = \beta'_k\quad (k\in\mathbb{Z}_2), \\
	&\pi'(\alpha_i) = \alpha_{i+1}\quad (i\in\mathbb{Z}_{2n+2}),\quad
	\pi'(\beta_k) = \beta_k,\quad
	\pi'(\beta'_k) = \beta'_{k+1}\quad (k\in\mathbb{Z}_2), \\
	&\rho(\alpha_i) = \alpha_{2n+2-i}\quad (i\in\mathbb{Z}_{2n+2}),\quad
	\rho(\beta_k) = \beta'_{k+1},\quad
	\rho(\beta'_k) = \beta_{k+1}\quad (k\in\mathbb{Z}_2),
\end{align*}
and the dependent variables as
\begin{align*}
	&\pi(\varphi_i) = \varphi_{i+1},\quad
	\pi'(\varphi_i) = \frac{\alpha_{i+1}}{\varphi_{i+1}}\quad (i\in\mathbb{Z}_{2n+2}), \\
	&\rho(\varphi_{2i}) = \frac{\alpha_{2n+2-2i}}{\varphi_{2n+2-2i}},\quad
	\rho(\varphi_{2i+1}) = \varphi_{2n+1-2i}\quad (i=0,\ldots,n).
\end{align*}

\begin{prop}
The simple reflections and the Dynkin diagram automorphisms satisfy fundamental relations
\begin{align*}
	&\pi^{2n+2} = 1,\quad
	\pi^2 = (\pi')^2,\quad
	\pi\,\pi' = \pi'\pi,\quad
	\rho^2 = 1,\quad
	\pi\,\rho = \rho\,(\pi')^{-1}, \\
	&r_i\,\pi = \pi\,r_{i+1},\quad
	s_k\,\pi = \pi\,s_{k+1},\quad
	s'_k\,\pi = \pi\,s'_k,\quad
	r_i\,\pi' = \pi'r_{i+1},\quad
	s_k\,\pi' = \pi's_k,\quad
	s'_k\,\pi' = \pi's'_{k+1}, \\
	&r_i\,\rho = \rho\,r_{2n+2-i},\quad
	s_k\,\rho = \rho\,s'_{k+1},
\end{align*}
for $i\in\mathbb{Z}_{2n+2}$ and $k\in\mathbb{Z}_2$.
\end{prop}

We can prove this proposition by direct calculations with
\begin{equation}\label{Eq:Fund_Rel_Quiver}
	\mu_i^2 = 1,\quad
	(j,k)\,\mu_i = \mu_{(j,k)(i)}\,(j,k)\quad (i,j,k=1,\ldots,4n+4).
\end{equation}
Hence the semi-direct product $\langle G,H,H'\rangle\rtimes\langle\pi,\pi',\rho\rangle$ can be regarded as an extended affine Weyl group of type $(A_{2n+1}+A_1+A_1)^{(1)}$.

\section{Generalized $q$-Painlev\'e VI systems}\label{Sec:Gen_q-PVI}

In Section \ref{Sec:Bir_Rep}, we have obtained the extended affine Weyl group $\langle G,H,H'\rangle\rtimes\langle\pi,\pi',\rho\rangle$.
In this section we give its abelian normal subgroup generated by translations, which provides a class of generalized $q$-Painlev\'e VI systems containing the known three systems.

Following definition of translations \eqref{Eq:Translation}, we define transformations $T_i,T'_i$ $(i\in\mathbb{Z}_{2n+2})$ and $U_k,U'_k$ $(k\in\mathbb{Z}_2)$ by
\begin{align*}
	&T_0 = r_1\ldots r_{2n+1}\,\pi,\quad
	T_i = r_{i+1}\ldots r_{2n+1}\,\pi\,r_1\ldots r_i\quad (i=1,\ldots,2n),\quad
	T_{2n+1} = \pi\,r_1\ldots r_{2n+1}, \\
	&T'_0 = r_1\ldots r_{2n+1}\,\pi',\quad
	T'_i = r_{i+1}\ldots r_{2n+1}\,\pi'r_1\ldots r_i\quad (i=1,\ldots,2n),\quad
	T'_{2n+1} = \pi'r_1\ldots r_{2n+1}, \\
	&U_0 = s_1\,\pi,\quad
	U_1 = \pi\,s_1,\quad
	U'_0 = s'_1\,\pi',\quad
	U'_1 = \pi's'_1.
\end{align*}
However all of them can't be regarded as translations.
As a matter of fact, each of $T_i$ and $T'_i$ acts on $\beta_k$ and $\beta'_k$ respectively as a permutation.
Also each of $U_k,U'_k$ acts on $\alpha_i$ as a cyclic permutation.
Therefore we consider compositions of those transformations as
\begin{align*}
	&\mathcal{T}_i = T_{i-1}^{-1}\,T_i = (T'_{i-1})^{-1}T'_i\quad (i\in\mathbb{Z}_{2n+2}),\quad
	\mathcal{U}_k = U_{k-1}^{-1}\,U_k,\quad
	\mathcal{U}'_k = (U'_{k-1})^{-1}U'_k\quad (k\in\mathbb{Z}_2), \\
	&\mathcal{V} = s_1\,T_0,\quad
	\mathcal{V}' = s'_1\,T'_0.
\end{align*}
Then they turn out to be translations acting on the parameters as
\begin{align*}
	&\mathcal{T}_j(\alpha_i) = q^{-\delta_{i,j-1}+2\delta_{i,j}-\delta_{i,j+1}}\alpha_i,\quad
	\mathcal{T}_j(\beta_k) = \beta_k,\quad
	\mathcal{T}_j(\beta'_k) = \beta'_k, \\
	&\mathcal{U}_l(\alpha_i) = \alpha_i,\quad
	\mathcal{U}_l(\beta_k) = q^{2\delta_{k,l}-2\delta_{k,l+1}}\beta_k,\quad
	\mathcal{U}_l(\beta'_k) = \beta'_k, \\
	&\mathcal{U}'_l(\alpha_i) = \alpha_i,\quad
	\mathcal{U}'_l(\beta_k) = \beta_k,\quad
	\mathcal{U}'_l(\beta'_k) = q^{2\delta_{k,l}-2\delta_{k,l+1}}\beta'_k, \\
	&\mathcal{V}(\alpha_i) = q^{\delta_{i,0}-\delta_{i,1}}\alpha_i,\quad
	\mathcal{V}(\beta_k) = q^{\delta_{k,0}-\delta_{k,1}}\beta_k,\quad
	\mathcal{V}(\beta'_k) = \beta'_k, \\
	&\mathcal{V}'(\alpha_i) = q^{\delta_{i,0}-\delta_{i,1}}\alpha_i,\quad
	\mathcal{V}'(\beta_k) = \beta_k,\quad
	\mathcal{V}'(\beta'_k) = q^{\delta_{k,0}-\delta_{k,1}}\beta'_k,
\end{align*}
for $i,j\in\mathbb{Z}_{2n+2}$ and $k,l\in\mathbb{Z}_2$.
Those translations generate an abelian normal subgroup of $\langle G,H,H'\rangle\rtimes\langle\pi,\pi',\rho\rangle$ with fundamental relations
\begin{align*}
	&r_i\,\mathcal{T}_j = \mathcal{T}_j\,\mathcal{T}_i^{\delta_{i,j-1}-2\delta_{i,j}+\delta_{i,j+1}}r_i,\quad
	s_k\,\mathcal{T}_j = \mathcal{T}_j\,s_k,\quad
	s'_k\,\mathcal{T}_j = \mathcal{T}_j\,s'_k, \\
	&\pi\,\mathcal{T}_j = \mathcal{T}_{j-1}\,\pi,\quad
	\pi'\,\mathcal{T}_j = \mathcal{T}_{j-1}\,\pi',\quad
	\rho\,\mathcal{T}_j = \mathcal{T}_{2n+2-j}\,\rho, \\
	&r_i\,\mathcal{U}_l = \mathcal{U}_l\,r_i,\quad
	s_k\,\mathcal{U}_l = \mathcal{U}_{l-1}\,s_k,\quad
	s'_k\,\mathcal{U}_l = \mathcal{U}_l\,s'_k, \\
	&\pi\,\mathcal{U}_l = \mathcal{U}_{l-1}\,\pi,\quad
	\pi'\,\mathcal{U}_l = \mathcal{U}_l\,\pi',\quad
	\rho\,\mathcal{U}_l = \mathcal{U}'_{l-1}\,\rho, \\
	&r_i\,\mathcal{U}'_l = \mathcal{U}'_l\,r_i,\quad
	s_k\,\mathcal{U}'_l = \mathcal{U}'_l\,s_k,\quad
	s'_k\,\mathcal{U}'_l = \mathcal{U}'_{l-1}\,s'_k, \\
	&\pi\,\mathcal{U}'_l = \mathcal{U}'_l\,\pi,\quad
	\pi'\,\mathcal{U}'_l = \mathcal{U}'_{l-1}\,\pi',\quad
	\rho\,\mathcal{U}'_l = \mathcal{U}_{l-1}\,\rho, \\
	&r_i\,\mathcal{V} = \mathcal{V}\,\mathcal{T}_i^{-\delta_{i,0}+\delta_{i,1}}r_i,\quad
	s_k\,\mathcal{V} = \mathcal{V}\,\mathcal{U}_1\,s_k,\quad
	s'_k\,\mathcal{V} = \mathcal{V}\,s'_k, \\
	&\pi\,\mathcal{V} = \mathcal{V}\,\mathcal{T}_0^{-1}\,\mathcal{U}_1\,\pi,\quad
	\pi'\,\mathcal{V} = \mathcal{V}\,\pi',\quad
	\rho\,\mathcal{V} = (\mathcal{V}')^{-1}\,\mathcal{T}_0\,\rho, \\
	&r_i\,\mathcal{V}' = \mathcal{V}'\,\mathcal{T}_i^{-\delta_{i,0}+\delta_{i,1}}r_i,\quad
	s_k\,\mathcal{V}' = \mathcal{V}'s_k,\quad
	s'_k\,\mathcal{V}' = \mathcal{V}'\,\mathcal{U}'_1\,s'_k, \\
	&\pi\,\mathcal{V}' = \mathcal{V}'\pi,\quad
	\pi'\,\mathcal{V}' = \mathcal{V}'\,\mathcal{T}_0^{-1}\,\mathcal{U}'_1\,\pi',\quad
	\rho\,\mathcal{V}' = \mathcal{V}^{-1}\,\mathcal{T}_0\,\rho,
\end{align*}
for $i,j\in\mathbb{Z}_{2n+2}$ and $k,l\in\mathbb{Z}_2$.
Note that the group of the translations is actually generated by $2n+3$ elements $\mathcal{T}_1,\ldots,\mathcal{T}_{2n},\mathcal{U}_1,\mathcal{V},\mathcal{V}'$.
Also note that the extended affine Weyl group excepting $\rho$ is decomposed into semi-direct product of the group of the translations and a reducible finite Weyl group as
\[
	\langle G,H,H'\rangle\rtimes\langle\pi,\pi'\rangle = \langle\mathcal{T}_1,\ldots,\mathcal{T}_{2n},\mathcal{U}_1,\mathcal{V},\mathcal{V}'\rangle\rtimes\langle r_1,\ldots,r_{2n+1},s_1,s'_1\rangle.
\]

In the following, we focus on a translation which has a factor $\pi^{-1}\pi'$ or $\pi^2=(\pi')^2$.
Specifically, we consider two types of translations
\[
	\tau_1 = \mathcal{U}_1^{-1}\,\mathcal{V}^{-1}\,\mathcal{V}' = U_1^{-1}U'_0 = s_1\,s'_1\,\pi^{-1}\pi',
\]
and
\begin{align*}
	\tau_2 &= \mathcal{T}_1^{-2n+1}\,\mathcal{T}_2^{-2n+2}\ldots\mathcal{T}_n^{-n}\,\mathcal{T}_{n+1}^{-n}\,\mathcal{T}_{n+2}^{-n+1}\ldots\mathcal{T}_{2n}^{-1}\,\mathcal{U}_1^{-n}\,\mathcal{V}^{-2n} \\
	&= T_n\,T_{2n+1} \\
	&= r_{n+1}\ldots r_{2n+1}\,r_0\ldots r_{n-1}\,r_{2n+1}\,r_0\ldots r_{2n-1}\,\pi^2.
\end{align*}
They act on the parameters as
\begin{align*}
	&\tau_1(\alpha_i) = \alpha_i\quad (i\in\mathbb{Z}_{2n+2}),\quad
	\tau_1(\beta_k) = q^{\delta_{k,0}-\delta_{k,1}}\beta_k,\quad
	\tau_1(\beta'_k) = q^{\delta_{k,0}-\delta_{k,1}}\beta'_k\quad (k\in\mathbb{Z}_2), \\
	&\tau_2(\alpha_i) = q^{-\delta_{i,0}+\delta_{i,n}-\delta_{i,n+1}+\delta_{i,2n+1}}\alpha_i\quad (i\in\mathbb{Z}_{2n+2}),\quad
	\tau_2(\beta_k) = \beta_k,\quad
	\tau_2(\beta'_k) = \beta'_k\quad (k\in\mathbb{Z}_2).
\end{align*}
In addition, we consider a translation
\[
	\mathcal{T}_1^{-n}\,\mathcal{T}_2^{-n}\,\mathcal{T}_3^{-n+1}\,\mathcal{T}_4^{-n+1}\ldots\mathcal{T}_{2n-1}^{-1}\,\mathcal{T}_{2n}^{-1}\,\mathcal{U}_1^{-m}\,\mathcal{V}^{-n-1} = T_1\,T_3\ldots T_{2n+1},
\]
for $n=2m-1$ and
\[
	\mathcal{T}_1^{-2n}\,\mathcal{T}_2^{-2n}\,\mathcal{T}_3^{-2n+2}\,\mathcal{T}_4^{-2n+2}\ldots\mathcal{T}_{2n-1}^{-2}\,\mathcal{T}_{2n}^{-2}\,\mathcal{U}_1^{-n-1}\,\mathcal{V}^{-2n-2} = T_1^2\,T_3^2\ldots T_{2n+1}^2,
\]
for $n=2m$.
It can be factorized into a power as
\[
	T_1\,T_3\ldots T_{2n+1} = (r_0\,r_2\ldots r_{2n}\,\pi)^{n+1}.
\]
Then we set
\[
	\tau_3 = (r_0\,r_2\ldots r_{2n}\,\pi)^2.
\]
It acts on the parameters as
\begin{align*}
	&\tau_3(\alpha_{2i}) = \frac{1}{\alpha_{2i+1}\,\alpha_{2i+2}\,\alpha_{2i+3}},\quad
	\tau_3(\alpha_{2i+1}) = \alpha_{2i+1}\,\alpha_{2i+2}\,\alpha_{2i+3}\,\alpha_{2i+4}\,\alpha_{2i+5}\quad (i=0,\ldots,n), \\
	&\tau_3(\beta_k) = \beta_k,\quad
	\tau_3(\beta'_k) = \beta'_k\quad (k\in\mathbb{Z}_2).
\end{align*}
Although $\tau_3$ isn't a translation in general, it is equivalent to the translation $\tau_2$ only in the case $n=1$.

\begin{thm}\label{Thm:Gen_q-PVI}
The translations $\tau_1,\tau_2$ and the transformation $\tau_3$ imply three types of generalized $q$-Painlev\'e VI systems as follows.
\begin{itemize}
\item
$\tau_1$ implies the $q$-Painlev\'e system $q$-$P_{(n+1,n+1)}$ arising from the $q$-DS hierarchy given in \cite{Suz2,Suz3}; see Theorem \ref{Thm:q-FST}.
\item
$\tau_2$ implies Sakai's $q$-Garnier system given in \cite{Sak2}; see Theorem \ref{Thm:q-Garnier}.
\item
$\tau_3$ implies Tsuda's $q$-Painlev\'e system arising from the $q$-LUC hierarchy given in Section 3.4 of \cite{T3}; see Theorem \ref{Thm:q-Tsuda}.
\end{itemize}
\end{thm}

As will be seen later, the origins and formulations of those three systems are quite different.
Therefore we divide Theorem \ref{Thm:Gen_q-PVI} into three and prove them individually.
Once we prove Theorem \ref{Thm:Gen_q-PVI}, we can handle three different systems in a unified way from a viewpoint of extended affine Weyl groups.

\begin{rem}
The translation $\mathcal{V}$ (resp. $\mathcal{V}'$) has a factor $\pi$ (resp. $\pi'$).
We conjecture that it implies a variation of the $q$-Garnier system given in Section 3.2.4 of \cite{NagY2}.
Besides, the translations $\mathcal{T}_i$ $(i\in\mathbb{Z}_{2n+2})$ and $\mathcal{U}_k,\mathcal{U}'_k$ $(k\in\mathbb{Z}_2)$ are products of the simple reflections.
Their explicit formulas are larger and more complicated than those of $\tau_1,\tau_2$.

\end{rem}

\section{$\tau_1$: $q$-Painlev\'e system $q$-$P_{(n+1,n+1)}$ arising from $q$-DS hierarchy}\label{Sec:q-FST}

The Drinfeld-Sokorov hierarchies were proposed as extensions of the Korteweg-de Vries hierarchy for the affine Lie algebras in \cite{DS}.
It is known that the DS hierarchies imply several continuous Painlev\'e equations and generalizations via operations called similarity reductions.
A $q$-analogue of the DS hierarchy of type $A^{(1)}_{2n+1}$ was proposed for the purpose of formulating a new generalized $q$-Painlev\'e VI system in \cite{Suz2,Suz3}.
As a result, we obtained a system of $q$-difference equations named $q$-$P_{(n+1,n+1)}$.
In this section we show that $q$-$P_{(n+1,n+1)}$ is equivalent to the action of the translation $\tau_1$ on $(\alpha_i,\beta_k,\beta'_k,\varphi_i)$.
Accordingly the affine Weyl group symmetry for $q$-$P_{(n+1,n+1)}$ is also derived from the extended affine Weyl group $\langle G,H,H'\rangle\rtimes\langle\pi,\pi',\rho\rangle$.

Let $a_i,b_i$ $(i=1,\ldots,n+1)$ be parameters, $t$ an independent variable and $f_i,g_i$ $(i=1,\ldots,n)$ dependent variables.
For the sake of simplicity, we use notations
\[
	b_0 = q\,b_{n+1},\quad
	f_0 = t,\quad
	g_0 = \frac{1}{q^{\frac{n-2}{2}}t\,g_1\ldots g_n}.
\]
Then $q$-$P_{(n+1,n+1)}$ is decribed as
\begin{equation}\label{Eq:q-FST}
	f_i\,\overline{f}_i = q\,t\,\frac{F_i\,F_{i+1}\left(\frac{b_i}{\overline{g}_i}-1\right)\left(\overline{g}_i-a_{i+1}\right)}{F_{n+1}F_1\left(\frac{b_0}{\overline{g}_0}-1\right)\left(\overline{g}_0-a_1\right)},\quad
	g_i\,\overline{g}_i = \frac{F_{i+1}\,G_i}{F_i\,G_{i+1}}\quad (i=1,\ldots,n),
\end{equation}
where
\begin{align*}
	F_i &= \sum_{j=1}^{i-1}f_j + t\sum_{j=i}^{n}f_j + t, \\
	G_i &= \sum_{j=i}^{n}\prod_{k=i}^{j-1}b_k\,a_{k+1}\frac{\prod_{l=j+1}^{n}g_l}{\prod_{l=1}^{j-1}g_l}f_j + q^{\frac n2}t\prod_{k=i}^{n}b_k\,a_{k+1} + q^{n-1}t\sum_{j=1}^{i-1}\prod_{k=0}^{j-1}b_k\,a_{k+1}\prod_{k=i}^{n}b_k\,a_{k+1}\frac{\prod_{l=j+1}^{n}g_l}{\prod_{l=1}^{j-1}g_l}f_j,
\end{align*}
for $i=1,\ldots,n+1$.
Here the forward and backward $q$-shifts are defined by
\[
	\overline{x(t)} = x(q\,t),\quad
	\underline{x(t)} = x(q^{-1}t).
\]
Note that $q$-$P_{(2,2)}$ coincides with the $q$-Painlev\'e VI equation.
System \eqref{Eq:q-FST} admits an affine Weyl group symmetry of type $A^{(1)}_{2n+1}$.
Let $\tilde{r}_0,\ldots,\tilde{r}_{2n+1}$ be birational transformations defined by
\begin{align*}
	&\tilde{r}_{2j}(b_j) = a_{j+1},\quad
	\tilde{r}_{2j}(a_{j+1}) = b_j,\quad
	\tilde{r}_{2j}(b_i) = b_i,\quad
	\tilde{r}_{2j}(a_{i+1}) = a_{i+1},\quad
	\tilde{r}_{2j}(t) = t, \\
	&\tilde{r}_{2j+1}(a_{j+1}) = b_{j+1}\quad
	\tilde{r}_{2j+1}(b_{j+1}) = a_{j+1},\quad
	\tilde{r}_{2j+1}(a_{i+1}) = a_{i+1},\quad
	\tilde{r}_{2j+1}(b_{i+1}) = b_{i+1},\quad
	\tilde{r}_{2j+1}(t) = t,
\end{align*}
for $i,j=0,\ldots,n$ with $i\neq j$,
\[
	\tilde{r}_{2j}(f_i) = f_i,\quad
	\tilde{r}_{2j}(g_i) = g_i\quad (i=1,\ldots,n,\ j=0,\ldots,n),
\]
and
\begin{align*}
	&\tilde{r}_1(f_i) = f_i\left(\frac{R_1^{a,a,a}}{R_1^{b,a,b}}\right)^{\delta_{i,1}}\left(\frac{R_1^{b,a,a}}{{R_1^{b,a,b}}}\right)^{1-\delta_{i,1}},\quad
	\tilde{r}_1(g_i) = g_i\left(\frac{R_1^{b,a,a}}{R_1^{b,b,b}}\right)^{\delta_{i,1}}, \\
	&\tilde{r}_{2j+1}(f_i) = f_i\left(\frac{R_{j+1}^{b,a,b}}{R_{j+1}^{b,a,a}}\right)^{\delta_{i,j}}\left(\frac{R_{j+1}^{a,a,a}}{R_{j+1}^{b,a,a}}\right)^{\delta_{i,j+1}},\quad
	\tilde{r}_{2j+1}(g_i) = g_i\left(\frac{R_{j+1}^{b,b,b}}{R_{j+1}^{b,a,a}}\right)^{\delta_{i,j}-\delta_{i,j+1}}, \\
	&\tilde{r}_{2n+1}(f_i) = f_i\left(\frac{R_{n+1}^{b,a,b}}{R_{n+1}^{a,a,a}}\right)^{\delta_{i,n}}\left(\frac{R_{n+1}^{b,a,a}}{R_{n+1}^{a,a,a}}\right)^{1-\delta_{i,n}},\quad
	\tilde{r}_{2n+1}(g_i) = g_i\left(\frac{R_{n+1}^{b,b,b}}{R_{n+1}^{b,a,a}}\right)^{\delta_{i,n}},
\end{align*}
for $i=1,\ldots,n$ and $j=1,\ldots,n-1$, where
\begin{align*}
	R_j^{\alpha,\beta,\gamma} &= (g_j-\alpha_j)\frac{1}{f_j} + \left(\beta_j\,\frac{b_{j-1}}{g_{j-1}}-\gamma_j\right)\frac{1}{f_{j-1}}\quad (j=1,\ldots,n), \\
	R_{n+1}^{\alpha,\beta,\gamma} &= \frac{1}{q}(g_0-q\,\alpha_{n+1}) + \left(\beta_{n+1}\,\frac{b_n}{g_n}-\gamma_{n+1}\right)\frac{1}{f_n}.
\end{align*}
Also let $\tilde{\pi}$ be a birational transformation defined by
\begin{align*}
	&\tilde{\pi}(a_i) = \frac{b_i}{q^{\rho_1}},\quad
	\tilde{\pi}(b_i) = \frac{a_{i+1}}{q^{\rho_1}}\quad (i=1,\ldots,n),\quad
	\tilde{\pi}(a_{n+1}) = \frac{b_{n+1}}{q^{\rho_1}},\quad
	\tilde{\pi}(b_{n+1}) = \frac{a_1}{q^{\rho_1+1}}, \\
	&\tilde{\pi}(\rho_1) = -\rho_1-\frac{1}{n+1},\quad
	\tilde{\pi}(t) = \frac{q^2}{t},
\end{align*}
where
\[
	q^{\rho_1} = (q^n\,a_1\,b_1\ldots a_{n+1}\,b_{n+1})^{\frac{1}{n+1}},
\]
and
\begin{align*}
	&\tilde{\pi}(f_i) = \frac{q^2}{t}\frac{(g_i\,R_i^{*}-b_{i+1}\,R_{i+1}^{*})(b_{i+1}-g_{i+1})\left(R_{i+1}^{*}+1-\frac tq\right)f_1}{(g_0\,R_0^{*}-b_1\,R_1^{*})(b_1-g_1)\left(R_1^{*}+1-\frac tq\right)f_{i+1}},\quad
	\tilde{\pi}(g_i) = \frac{a_{i+1}}{q^{\rho_1}}\frac{b_{i+1}\,R_{i+1}^{*}}{g_i\,R_i^{*}}, \\
	&\tilde{\pi}(f_n) = q\,\frac{(g_n\,R_n^{*}-b_{n+1}\,R_0^{*})(b_0-g_0)\left(R_0^{*}+1-\frac tq\right)f_1}{(g_0\,R_0^{*}-b_1\,R_1^{*})(b_1-g_1)\left(R_1^{*}+1-\frac tq\right)f_0},\quad
	\tilde{\pi}(g_n) = \frac{a_{n+1}}{q^{\rho_1}}\frac{b_{n+1}\,R_0^{*}}{g_n\,R_n^{*}},
\end{align*}
for $i=1,\ldots,n-1$, where
\[
	R_i^{*} = \frac{f_i}{b_i-g_i}\left(\sum_{j=0}^{i-1}\frac{\frac tq\left(1-\frac{a_{j+1}}{g_j}\right)(b_j-g_j)}{f_j}+\frac{\left(\frac tq-\frac{a_{i+1}}{g_i}\right)(b_i-g_i)}{f_i}+\sum_{j=i+1}^{n}\frac{\left(1-\frac{a_{j+1}}{g_j}\right)(b_j-g_j)}{f_j}\right),
\]
for $i=0,\ldots,n$.
Then system \eqref{Eq:q-FST} is invariant under the actions of $\tilde{r}_0,\ldots,\tilde{r}_{2n+1},\tilde{\pi}$.
Moreover, a group of birational transformations $\langle\tilde{r}_0,\ldots,\tilde{r}_{2n+1},\tilde{\pi}\rangle$ is isomorphic to an extended affine Weyl group $W(A^{(1)}_{2n+1})\rtimes\langle\tilde{\pi}\rangle$.

Those formulas were derived from the $q$-DS hierarchy.
We now obtain them from the extended affine Weyl group $\langle G,H,H'\rangle\rtimes\langle\pi,\pi',\rho\rangle$.

\begin{thm}\label{Thm:q-FST}
If we set
\[
	\frac{a_i}{b_i} = \alpha_{2i-1}\quad (i=1,\ldots,n+1),\quad
	\frac{b_i}{a_{i+1}} = \alpha_{2i}\quad (i=1,\ldots,n),\quad
	q^n\,t\,\prod_{i=1}^{n+1}a_i\,b_i = \beta_0,\quad
	t = \beta'_0,
\]
and
\[
	f_i = \frac{1+\varphi_{2i}}{1+\varphi_0}\,\prod_{j=i}^{n}\varphi_{2j+1}\,\varphi_{2j+2},\quad
	\frac{g_i}{b_i} = -\frac{1}{\varphi_{2i}}\quad (i=1,\ldots,n),
\]
then system \eqref{Eq:q-FST} is equivalent to the action of the translation $\tau_1$ on $(a_i,b_i,t,f_i,g_i)$.
Moreover, the birational transformations $\tilde{r}_0,\ldots,\tilde{r}_{2n+1},\tilde{\pi}$ are equivalent to $r_0,\ldots,r_{2n+1},s'_1\,\pi$ respectively.
\end{thm}

We prove this theorem in Appendix \ref{App:Prf_q-FST}.
Note that we can identify a tuple of parameters and variables $(a_i,b_i,t,f_i,g_i)$ with $(\alpha_i,\beta_k,\beta'_k,\varphi_i)$ thanks to equations \eqref{Proof:Gen_q-PVI_par} and \eqref{Proof:Gen_q-PVI_g_3} given throughout the proof.

In the last, we formulate the birational transformations $\tilde{r}_0,\ldots,\tilde{r}_{2n+1},\tilde{\pi}^2$ in terms of $(2n+2)\times(2n+2)$ matrices following \cite{Suz2,Suz3}.
This formulation is called a Lax form and will be used in Section \ref{Sec:q-Garnier} to derive the $q$-Garnier system.
Note that
\[
	\tilde{\pi}^2 = s'_1\,\pi\,s'_1\,\pi = \pi^2 = (\pi')^2.
\]
Let $h$ be a dependent variable called a gauge freedom satisfying
\[
	\overline{h} = -\frac{F_{n+1}\,F_1\left(\frac{b_0}{\overline{g}_0}-1\right)\left(\overline{g}_0-a_1\right)}{t\,(t-1)^2\,\overline{g}_0},
\]
and
\begin{align*}
	&\tilde{r}_1(h) = h + \frac{(a_1-b_1)\,h}{t\,R_1^{b,a,a}},\quad
	\tilde{r}_{2n+1}(h) = h - \frac{(a_{n+1}-b_{n+1})\,h}{t\,R_{n+1}^{b,a,a}},\quad
	\tilde{r}_j(h) = h\quad (j\neq1,2n+1), \\
	&\tilde{\pi}^2(h) = \frac{f_1\,h}{t}.
\end{align*}
We consider a system of linear $q$-difference equations
\begin{equation}\label{qFST_Lax_1}
	T_{q,z}^{-1}(\psi) = M\,\psi,
\end{equation}
with a $(2n+2)\times(2n+2)$ matrix
\[
	M = \begin{pmatrix}M_{1,1}&M_{1,2}&&&&O\\&M_{2,2}&M_{2,3}&&&\\&&M_{3,3}&&\\&&&\ddots&&\\&O&&&M_{n,n}&M_{n,n+1}\\z\,M_{n+1,1}&&&&&M_{n+1,n+1}\end{pmatrix},
\]
where
\begin{align*}
	&M_{i,i} = \begin{pmatrix}a_i&\frac{t\,R_i^{b,a,a}}{h}\\0&b_i\end{pmatrix}\quad (i=1,\ldots,n+1), \\
	&M_{i,i+1} = \begin{pmatrix}-1&0\\\frac{f_i\,h}{t}&-1\end{pmatrix}\quad (i=1,\ldots,n),\quad
	M_{n+1,1} = \begin{pmatrix}-t&0\\h&-1\end{pmatrix}.
\end{align*}
Here the forward and backward $q$-shifts are defined by
\[
	T_{q,z}(x(z)) = x(q\,z),\quad
	T_{q,z}^{-1}(x(z)) = x(q^{-1}z).
\]
We also consider gauge transformations for system \eqref{qFST_Lax_1}
\begin{equation}\label{qFST_Lax_2}
	\tilde{r}_j(\psi) = \Gamma_j\,\psi\quad (j=0,\ldots,2n+1),\quad
	\tilde{\pi}^2(\psi) = \Pi\,\psi,
\end{equation}
where
\begin{align*}
	&\Gamma_0 = I - \frac{b_0-a_1}{q\,h}\,z^{-1}E_{1,2n+2},\quad
	\Gamma_{2j} = I - \frac{(b_j-a_{j+1})\,t}{f_j\,h}\,E_{2j+1,2j}\quad (j=1,\ldots,n), \\
	&\Gamma_{2j+1} = I - \frac{(a_{j+1}-b_{j+1})\,h}{t\,R_j^{b,a,a}}\,E_{2j+2,2j+1}\quad (j=0,\ldots,n), \\
	&\Pi = z^{-\frac{1}{n+1}}\left(\sum_{i=1}^{n}q^{\frac{i-1}{n+1}}E_{2i-1,2i+1}+\sum_{i=1}^{n}q^{\frac{i-1}{n+1}}\,t^{\frac{1}{n+1}}E_{2i,2i+2}+q^{\frac{n}{n+1}}\,t\,z\,E_{2n+1,1}+q^{\frac{n}{n+1}}\,t^{\frac{1}{n+1}}\,z\,E_{2n+2,2}\right).
\end{align*}
Here the symbol $I$ stands for the identity matrix and $E_{i,j}$ a $(2n+2)\times(2n+2)$ matrix with $1$ in $(i,j)$-th entry and $0$ elsewhere.
Then the compatibility condition of \eqref{qFST_Lax_1} and \eqref{qFST_Lax_2}
\[
	\tilde{r}_j(M)\,\Gamma_j = T_{q,z}^{-1}(\Gamma_j)\,M\quad (j=0,\ldots,2n+1),\quad
	\tilde{\pi}^2(M)\,\Pi = T_{q,z}^{-1}(\Pi)\,M,\quad
\]
implies the actions of birational transformations $\tilde{r}_0,\ldots,\tilde{r}_{2n+1},\tilde{\pi}^2$ on $(a_i,b_i,t,f_i,g_i)$ defined above.

\begin{rem}
In addition to the simple reflections $r_i=\tilde{r}_i$ $(i\in\mathbb{Z}_{2n+2})$ and the Dynkin diagram automorphism $\pi^2=(\pi')^2=\tilde{\pi}^2$, we have formulated the translation $\tau_1$ in a framework of a Lax form in \cite{Suz2,Suz3}.
However we haven't done the same for the other transformations $s_0,s_1,s'_0,s'_1,\pi,\pi',\rho$ or the mutations $\mu_1,\ldots,\mu_{4n+4}$ yet.
It is a future problem.
\end{rem}

\section{Sakai's $q$-Garnier system}\label{Sec:q-Garnier}

The $q$-Garnier system was proposed by Sakai as a $q$-analogue of the isomonodromy deformation of the second order Fuchsian differential equation with $n+3$ regular singular points in \cite{Sak2}.
It is uniquely determined as the compatibility condition of a system of linear $q$-difference equations with $2\times2$ matrices, whose detail is given in Theorem \ref{Thm:q-Garnier}.
Afterward, it was investigated by a Pad\'e method in \cite{NagY1,NagY2}.
However, in both previous works, the explicit formulas of the $q$-Garnier system are large and complicated.
Hence it is hard to verify that the system is equivalent to the action of the translation $\tau_2$ on $(\alpha_i,\beta_k,\beta'_k,\varphi_i)$.
Therefore, in this section, we show the equivalency at the level of Lax forms.

Thanks to Theorem \ref{Thm:q-FST}, we can give the translation $\tau_2$ as a composition of the transformations $\tilde{r}_0,\ldots,\tilde{r}_{2n+1},\tilde{\pi}^2$ by
\[
	\tau_2 = \tilde{r}_{n+1}\ldots\tilde{r}_{2n+1}\,\tilde{r}_0\ldots\tilde{r}_{n-1}\,\tilde{r}_{2n+1}\,\tilde{r}_0\ldots\tilde{r}_{2n-1}\,\tilde{\pi}^2.
\]
In Section \ref{Sec:q-FST} those transformations have been formulated in a framework of the Lax form.
This fact allows us to formulate a system of linear $q$-difference equations with $(2n+2)\times(2n+2)$ matrices
\begin{equation}\label{Eq:q-Garnier_Lax_1}
	T_{q,z}^{-1}(\psi) = M\,\psi,\quad
	\tau_2(\psi) = \Gamma\,\psi,
\end{equation}
whose compatibility condition $\tau_2(M)\,\Gamma=T_{q,z}^{-1}(\Gamma)\,M$ implies the action of $\tau_2$ on $(a_i,b_i,t,f_i,g_i)$.
Here the matrix $\Gamma$ is given by
\begin{align*}
	\Gamma &= \tilde{r}_{n+2}\ldots\tilde{r}_{2n+1}\,\tilde{r}_0\ldots\tilde{r}_{n-1}\,\tilde{r}_{2n+1}\,\tilde{r}_0\ldots\tilde{r}_{2n-1}\,\tilde{\pi}^2(\Gamma_{n+1}) \\
	&\quad\times \tilde{r}_{n+3}\ldots\tilde{r}_{2n+1}\,\tilde{r}_0\ldots\tilde{r}_{n-1}\,\tilde{r}_{2n+1}\,\tilde{r}_0\ldots\tilde{r}_{2n-1}\,\tilde{\pi}^2(\Gamma_{n+2}) \\
	&\quad\times \ldots \\
	&\quad\times \tilde{r}_{2n-1}\,\tilde{\pi}^2(\Gamma_{2n-2}) \\
	&\quad\times \tilde{\pi}^2(\Gamma_{2n-1}) \\
	&\quad\times \Pi.
\end{align*}
It is also described as
\[
	\Gamma = z^{-\frac{1}{n+1}}\begin{pmatrix}\Gamma_{1,1}&\Gamma_{1,2}&&&&O\\&\Gamma_{2,2}&\Gamma_{2,3}&&&\\&&\Gamma_{3,3}&&\\&&&\ddots&&\\&O&&&\Gamma_{n,n}&\Gamma_{n,n+1}\\z\,\Gamma_{n+1,1}&&&&&\Gamma_{n+1,n+1}\end{pmatrix},
\]
where
\begin{align*}
	&\Gamma_{i,i} = \begin{pmatrix}\gamma_{2i-1,2i-1}&\gamma_{2i-1,2i}\\0&\gamma_{2i,2i}\end{pmatrix}\quad (i=1,\ldots,n+1), \\
	&\Gamma_{i,i+1} = \begin{pmatrix}q^{\frac{i-1}{n+1}}&0\\\gamma_{2i,2i+1}&q^{\frac{i-1}{n+1}}\,t^{\frac{1}{n+1}}\end{pmatrix}\quad (i=1,\ldots,n),\quad
	\Gamma_{n+1,1} = \begin{pmatrix}q^{\frac{n}{n+1}}\,t&0\\\gamma_{2n+2,1}&q^{\frac{n}{n+1}}\,t^{\frac{1}{n+1}}\end{pmatrix}.
\end{align*}
Each of $\gamma_{i,j}$ is rational in the entries of the matrix $M$.
We don't give its explicit formula here.
We transform system \eqref{Eq:q-Garnier_Lax_1} to that with $2\times2$ matrices, which determines the $q$-Garnier system.
Note that the translation $\tau_2$ acts on the parameters and the independent variable as
\begin{align*}
	&\tau_2(a_i) = q^{\frac{1}{n+1}}\,a_i\quad (i=1,\ldots,n+1), \\
	&\tau_2(b_i) = q^{\frac{1}{n+1}}\,b_i\quad (i=1,\ldots,n,\ i\neq m),\quad
	\tau_2(b_m) = q^{-\frac{n}{n+1}}\,b_m,\quad
	\tau_2(b_{n+1}) = q^{-\frac{n}{n+1}}\,b_{n+1}, \\
	&\tau_2(t) = t,
\end{align*}
for $n=2m-1$ and
\begin{align*}
	&\tau_2(a_i) = q^{\frac{1}{n+1}}\,a_i\quad (i=1,\ldots,n+1,\ i\neq m),\quad
	\tau_2(a_m) = q^{-\frac{n}{n+1}}\,a_m \\
	&\tau_2(b_i) = q^{\frac{1}{n+1}}\,b_i\quad (i=1,\ldots,n),\quad
	\tau_2(b_{n+1}) = q^{-\frac{n}{n+1}}\,b_{n+1}, \\
	&\tau_2(t) = t,
\end{align*}
for $n=2m$.

We consider a gauge transformation
\[
	\hat{\psi} = z^{\log_qa_1}\,\psi.
\]
Then system \eqref{Eq:q-Garnier_Lax_1} is transformed to
\begin{equation}\label{Eq:q-Garnier_Lax_2}
	T_{q,z}^{-1}(\hat{\psi}) = \frac{1}{a_1}\,M\,\hat{\psi},\quad
	\tau_2(\hat{\psi}) = z^{\frac{1}{n+1}}\,\Gamma\,\hat{\psi}.
\end{equation}
We next consider a $q$-Laplace transformation
\[
	z\,\hat{\psi} \to T_{q,z}(\phi),\quad
	T_{q,z}^{-1}(\hat{\psi}) \to z\,\phi.
\]
Then system \eqref{Eq:q-Garnier_Lax_2} is transformed to
\begin{equation}\label{Eq:q-Garnier_Lax_3}
	\frac{1}{a_1}\,M_1\,T_{q,z}(\phi) = \left(z\,I-\frac{1}{a_1}\,M_0\right)\phi,\quad
	\tau_2(\phi) = z^{\frac{1}{n+1}}\,\Gamma_0\,\phi + z^{\frac{1}{n+1}}\,\Gamma_1\,T_{q,z}(\phi),
\end{equation}
where
\[
	M(z) = M_0 + z\,M_1,\quad
	\Gamma(z) = \Gamma_0 + z\,\Gamma_1.
\]
System \eqref{Eq:q-Garnier_Lax_3} is rewritten into
\begin{align*}
	T_{q,z}\begin{pmatrix}\phi_1\\\phi_2\end{pmatrix} &= \hat{M}_{n+1,1}^{-1}(z\,I-\hat{M}_{n+1,n+1})\begin{pmatrix}\phi_{2n+1}\\\phi_{2n+2}\end{pmatrix}, \\
	\begin{pmatrix}\phi_{2i+1}\\\phi_{2i+2}\end{pmatrix} &= \hat{M}_{i,i+1}^{-1}(z\,I-\hat{M}_{i,i})\begin{pmatrix}\phi_{2i-1}\\\phi_{2i}\end{pmatrix}\quad (i=1,\ldots,n), \\
	\tau_2\begin{pmatrix}\phi_{2i-1}\\\phi_{2i}\end{pmatrix} &= \Gamma_{i,i}\begin{pmatrix}\phi_{2i-1}\\\phi_{2i}\end{pmatrix} + \Gamma_{i,i+1}\begin{pmatrix}\phi_{2i+1}\\\phi_{2i+2}\end{pmatrix}\quad (i=1,\ldots,n), \\
	\tau_2\begin{pmatrix}\phi_{2n+1}\\\phi_{2n+2}\end{pmatrix} &= \Gamma_{n+1,n+1}\begin{pmatrix}\phi_{2n+1}\\\phi_{2n+2}\end{pmatrix} + z\,\Gamma_{n+1,1}\,T_{q,z}\begin{pmatrix}\phi_1\\\phi_2\end{pmatrix},
\end{align*}
where
\[
	\hat{M}_{i,j} = \frac{1}{a_1}\,M_{i,j},\quad
	\phi = \begin{pmatrix}\phi_1\\\vdots\\\phi_{2n+2}\end{pmatrix},
\]
from which we obtain
\begin{equation}\label{Eq:q-Garnier_Lax_4}
	T_{q,z}(\Phi) = \mathcal{A}\,\Phi,\quad
	\tau_2(\Phi) = \mathcal{B}\,\Phi,\quad
	\Phi = \begin{pmatrix}\phi_1\\\phi_2\end{pmatrix},
\end{equation}
where
\begin{align*}
	\mathcal{A} &= \hat{M}_{n+1,1}^{-1}\left(z\,I-\hat{M}_{n+1,n+1}\right)\hat{M}_{n,n+1}^{-1}\left(z\,I-\hat{M}_{n,n}\right)\ldots\hat{M}_{1,2}^{-1}\left(z\,I-\hat{M}_{1,1}\right), \\
	\mathcal{B} &= \Gamma_{1,1} + \Gamma_{1,2}\,\hat{M}_{1,2}^{-1}\left(z\,I-\hat{M}_{1,1}\right).
\end{align*}

\begin{thm}\label{Thm:q-Garnier}
The compatibility condition of \eqref{Eq:q-Garnier_Lax_1} implies that of \eqref{Eq:q-Garnier_Lax_4} $\tau_2(\mathcal{A})\,\mathcal{B}=T_{q,z}(\mathcal{B})\,\mathcal{A}$.
Moreover, the compatibility condition of \eqref{Eq:q-Garnier_Lax_4} is equivalent to an inverse direction of the $q$-Garnier system.
In other words, the matrices $\mathcal{A}$ and $\mathcal{B}$ satisfy the following properties.
\begin{enumerate}
\item
$\mathcal{A}(z)=\mathcal{A}_0+z\,\mathcal{A}_1+\ldots+z^{n+1}\mathcal{A}_{n+1}$.
\smallskip
\item
$\mathcal{A}_0\sim\begin{pmatrix}q^{-\frac n2}\,t^{-1}&0\\0&q^{\frac n2}\,a_1\,b_1\ldots a_{n+1}\,b_{n+1}\end{pmatrix}$, $\mathcal{A}_{n+1}\sim\begin{pmatrix}(-a_1)^{n+1}t^{-1}&0\\0&(-a_1)^{n+1}\end{pmatrix}$.
\smallskip
\item
$\det\mathcal{A}=\frac{a_1^{2n+2}}{t}\,(z-1)\left(z-\frac{b_1}{a_1}\right)\left(z-\frac{a_2}{a_1}\right)\left(z-\frac{b_2}{a_1}\right)\ldots\left(z-\frac{a_{n+1}}{a_1}\right)\left(z-\frac{b_{n+1}}{a_1}\right)$.
\smallskip
\item
$\det\mathcal{B}=\left\{\begin{array}{ll}t^{\frac{1}{n+1}}\,a_1^2\left(z-\frac{b_m}{a_1}\right)\left(z-\frac{b_{n+1}}{a_1}\right)&(n=2m-1)\\[4pt]t^{\frac{1}{n+1}}\,a_1^2\left(z-\frac{a_m}{a_1}\right)\left(z-\frac{b_{n+1}}{a_1}\right)&(n=2m)\end{array}\right.$.
\end{enumerate}
\end{thm}

We prove this theorem in Appendix \ref{App:Prf_q-Garnier}.
Hence the action of $\tau_2$ on $(a_i,b_i,t,f_i,g_i)$ is equivalent to an inverse direction of the $q$-Garnier system.

\begin{rem}
The $q$-Garnier system is a $q$-analogue of a system of partial differential equations and hence has multiple directions of discrete time evolutions.
All directions of the $q$-Garnier system are given by translations $T_i\,T_j$ $(i,j\in\mathbb{Z}_{2n+2},\ i\neq j)$.
\end{rem}

\begin{rem}
The transformation from \eqref{Eq:q-Garnier_Lax_2} to \eqref{Eq:q-Garnier_Lax_4} is suggested by \cite{NagY2}, in which a connection between the $q$-Garnier system and $q$-$P_{(n+1,n+1)}$ is clarified at the level of Lax forms with the aid of a duality in a reduction of the $q$-KP hierarchy.
\end{rem}

\begin{rem}
We can regard $(f_i,g_i)$ as new coordinates of the $q$-Garnier system.
However it is hard to give a transformation between $(f_i,g_i)$ and the coordinates given in \cite{Sak2} or \cite{NagY1,NagY2}.
The transformation is complicated and sometimes algebraic.
\end{rem}

\section{Tsuda's $q$-Painlev\'e system arising from $q$-LUC hierarchy}\label{Sec:q-Tsuda}

The UC hierarchy was proposed by Tsuda as a generalization of the KP hierarchy for the irreducible characters of the general linear group appearing as solutions of the hierarchies in \cite{T1}.
More precisely, the solutions of the KP hierarchy and the UC hierarchy are described in terms of the Schur function and the universal character respectively.
Similarly as the DS hierarchies, the UC hierarchy implies several continuous Painlev\'e equations and generalizations via similarity reductions.
Tsuda also proposed a $q$-analogue of the UC hierarchy named the lattice $q$-UC hierarchy in \cite{T2} and investigated a connection between the hierarchy and $q$-Painlev\'e systems in \cite{T3}.
As a result, he obtained a system of $q$-difference equations
\begin{equation}\begin{split}\label{Eq:q-LUC}
	\overline{f}_{i,n-i+2k} &= \frac{c_{i,n-i+2k}}{c_{i,n-i+2k+1}}\frac{(g_{i+1,n-i+2k}-\alpha)(g_{i,n-i+2k+1}-c_{i,n-i+2k+1}\,\beta)}{(g_{i,n-i+2k+1}-\alpha)(g_{i+1,n-i+2k}-c_{i+1,n-i+2k}\,\beta)}\,f_{i+1,n-i+2k+1}, \\
	\overline{g}_{i,n-i+2k-1} &= \frac{c_{i+1,n-i+2k-1}}{c_{i+1,n-i+2k}}\frac{\left(\overline{f}_{i+1,n-i+2k-1}-q^{\frac{1}{n+1}}\,\gamma\right)\left(\overline{f}_{i,n-i+2k}-c_{i,n-i+2k}\,\delta\right)}{\left(\overline{f}_{i,n-i+2k}-q^{\frac{1}{n+1}}\,\gamma\right)\left(\overline{f}_{i+1,n-i+2k-1}-c_{i+1,n-i+2k-1}\,\delta\right)}\,g_{i+1,n-i+2k},
\end{split}\end{equation}
where the indices $i$ and $j$ of $c_{i,j},f_{i,j},g_{i,j}$ are congruent modulo $n+1$, with constraints
\begin{align*}
	&\alpha\,\delta = \beta\,\gamma,\quad
	\frac{c_{i,n-i+2k}\,c_{i+1,n-i+2k+1}}{c_{i+1,n-i+2k}\,c_{i,n-i+2k+1}} = 1,\quad
	\frac{c_{i,n-i+2k-1}\,c_{i+1,n-i+2k}}{c_{i+1,n-i+2k-1}\,c_{i,n-i+2k}} = 1\quad (i,k=0,\ldots,n), \\
	&\prod_{i=0}^{n}f_{i,n-i+2k} = 1,\quad
	\prod_{i=0}^{n}g_{i,n-i+2k-1} = 1\quad (k=0,\ldots,n).
\end{align*}
Note that system \eqref{Eq:q-LUC} coincides with the $q$-Painlev\'e VI equation in the case $n=1$.
In this section we show that system \eqref{Eq:q-LUC} is equivalent to the action of the transformation $\tau_3$ on $(\alpha_i,\beta_k,\beta'_k,\varphi_i)$.

Let ${\boldsymbol t}=(t_0,t_1)$ be a $2$-tuple of independent variables defined by
\[
	t_0 = \prod_{i=0}^{n}\frac{1}{\varphi_{2i}} = \left(\frac{1}{\beta_0\,\beta'_0}\prod_{i=0}^{n}\alpha_{2i+1}\right)^{\frac12},\quad
	t_1 = \prod_{i=0}^{n}\frac{\alpha_{2i+1}}{\varphi_{2i+1}} = \left(\frac{\beta_0}{\beta'_0}\prod_{i=0}^{n}\alpha_{2i+1}\right)^{\frac12}.
\]
Then the transformation $\tau_3$ acts on ${\boldsymbol t}$ as
\[
	\tau_3({\boldsymbol t}) = q\,{\boldsymbol t} = (q\,t_0,q\,t_1).
\]
Also let ${\boldsymbol c}=(c_0,\ldots,c_{2n+1})$ be a $(2n+2)$-tuple of parameters defined by
\[
	c_{2i} = t_0^{\frac{1}{n+1}}t_1^{\frac{1}{n+1}}\alpha_{2i},\quad
	c_{2i+1} = t_0^{-\frac{1}{n+1}}t_1^{-\frac{1}{n+1}}\alpha_{2i+1}\quad (i=0,\ldots,n).
\]
Then ${\boldsymbol c}$ is invariant under an iterative action of $\tau_3$ as
\[
	\tau_3^m({\boldsymbol c}) = {\boldsymbol c},
\]
for $n=2m-1$ and
\[
	\tau_3^{n+1}({\boldsymbol c}) = {\boldsymbol c},
\]
for $n=2m$.
We now regard $\varphi_i$ $(i\in\mathbb{Z}_{2n+2})$ as variables depending on $({\boldsymbol t},{\boldsymbol c})$ and denote them by $\varphi_i({\boldsymbol t},{\boldsymbol c})$.
The transformation $\tau_3$ acts on $\varphi_i({\boldsymbol t},{\boldsymbol c})$ as
\[
	\tau_3(\varphi_i({\boldsymbol t},{\boldsymbol c})) = \varphi_i(q\,{\boldsymbol t},\tau_3({\boldsymbol c})).
\]
We define a set of dependent variables and parameters satisfying system \eqref{Eq:q-LUC} with the aid of the action of $\tau_3$ on ${\boldsymbol c}$.

\begin{thm}\label{Thm:q-Tsuda}
If we set
\begin{align*}
	&\alpha = -t_1^{\frac{1}{n+1}},\quad
	\beta = -t_0^{-\frac{1}{n+1}},\quad
	\gamma = -t_0^{\frac{1}{n+1}},\quad
	\delta = -t_1^{-\frac{1}{n+1}}, \\
	&c_{i,n-i+2k} = \frac{\tau_3^{-k}(c_{2i})}{q^{\frac{1}{n+1}}},\quad
	c_{i,n-i+2k-1} = \frac{1}{\tau_3^{-k}(c_{2i+1})}\quad (i,k=0,\ldots,n),
\end{align*}
and
\begin{align*}
	&f_{i,n-i+2k} = t_0^{\frac{1}{n+1}}\varphi_{2i}({\boldsymbol t},\tau_3^{-k}({\boldsymbol c})),\quad
	g_{i,n-i+2k-1} = \frac{t_0^{-\frac{1}{n+1}}\varphi_{2i+1}({\boldsymbol t},\tau_3^{-k}({\boldsymbol c}))}{\tau_3^{-k}(c_{2i+1})}, \\
	&\overline{f}_{i,n-i+2k} = q^{\frac{1}{n+1}}t_0^{\frac{1}{n+1}}\varphi_{2i}(q\,{\boldsymbol t},\tau_3^{-k}({\boldsymbol c})),\quad
	\overline{g}_{i,n-i+2k-1} = \frac{q^{-\frac{1}{n+1}}t_0^{-\frac{1}{n+1}}\varphi_{2i+1}(q\,{\boldsymbol t},\tau_3^{-k}({\boldsymbol c}))}{\tau_3^{-k}(c_{2i+1})},
\end{align*}
for $i,k=0,\ldots,n$, then they satisfy system \eqref{Eq:q-LUC}.
\end{thm}

We can prove this theorem by a direct calculation with
\begin{align*}
	\tau_3(\varphi_{2i}) &= \frac{1}{\alpha_{2i+2}\,\alpha_{2i+3}}\frac{(1+\varphi_{2i+1})(\alpha_{2i+3}+\varphi_{2i+3})}{(1+\varphi_{2i+3})(\alpha_{2i+1}+\varphi_{2i+1})}\,\varphi_{2i+2}, \\
	\tau_3(\varphi_{2i+1}) &= \alpha_{2i+1}\,\alpha_{2i+2}\,\frac{\left\{1+\tau_3(\varphi_{2i+2})\right\}\left\{\frac{1}{\alpha_{2i+1}\,\alpha_{2i+2}\,\alpha_{2i+3}}+\tau_3(\varphi_{2i})\right\}}{\left\{1+\tau_3(\varphi_{2i})\right\}\left\{\frac{1}{\alpha_{2i+3}\,\alpha_{2i+4}\,\alpha_{2i+5}}+\tau_3(\varphi_{2i+2})\right\}}\,\varphi_{2i+3},
\end{align*}
for $i=0,\ldots,n$.

\begin{rem}
Thanks to Theorem \ref{Thm:q-FST}, we can rewrite the transformation $\tau_3$ as
\[
	\tau_3 = \tilde{r}_0\,\tilde{r}_2\ldots\tilde{r}_{2n}\,\tilde{r}_{2n+1}\,\tilde{r}_1\ldots\tilde{r}_{2n-1}\,\tilde{\pi}^2.
\]
Hence we can derive a Lax form for system \eqref{Eq:q-LUC} in a similar manner as Section \ref{Sec:q-Garnier}.
We don't give its detail here.
\end{rem}

\begin{rem}
System \eqref{Eq:q-LUC} is actually closed in $2n+2$ dependent variables $f_{i,n-i},g_{i,n-i-1}$ $(i=0,\ldots,n)$.
More precisely, each of $f_{i,n-i}(q^m{\boldsymbol t}),g_{i,n-i-1}(q^m{\boldsymbol t})$ (resp. $f_{i,n-i}(q^{n+1}{\boldsymbol t}),g_{i,n-i-1}(q^{n+1}{\boldsymbol t})$) is rational in dependent variables $f_{i,n-i}({\boldsymbol t}),g_{i,n-i-1}({\boldsymbol t})$ for $n=2m-1$ (resp. $n=2m$).
This phenomenon corresponds to the fact that $\tau_3^m$ (resp. $\tau_3^{n+1}$) is the translation for $n=2m-1$ (resp. $n=2m$).
\end{rem}

\appendix

\section{Actions of simple reflections $s_k,s'_k$}\label{App:Birat_Transf}

We investigate the action of the simple reflection $s_0$.
The actions of the other simple reflections can be obtained from $s_1=\pi^{-1}s_0\,\pi$ and $s'_k=\rho\,s_{k+1}\,\rho$.

Let us start with the quiver $Q$ given in Figure \ref{Fig:Gen_q-PVI}.
The corresponding skew-symmetric matrix $\Lambda=\left(\lambda_{i,j}\right)_{i,j}$ is given by
\begin{equation}\begin{split}\label{Eq:Quiver_Matrix}
	\Lambda &= X_{1,4n+3} - X_{1,4n+4} - X_{1,2n+3} + X_{1,2n+4} - X_{2,4n+3} + X_{2,4n+4} + X_{2,2n+3} - X_{2,2n+4} \\
	&\quad + \sum_{i=1}^{n}(X_{2i+1,2i+2n+1} - X_{2i+1,2i+2n+2} - X_{2i+1,2i+2n+3} + X_{2i+1,2i+2n+4}) \\
	&\quad + \sum_{i=1}^{n}(-X_{2i+2,2i+2n+1} + X_{2i+2,2i+2n+2} + X_{2i+2,2i+2n+3} - X_{2i+2,2i+2n+4}),
\end{split}\end{equation}
where $X_{i,j}=E_{i,j}-E_{j,i}$ and $E_{i,j}$ is a $(4n+4)\times(4n+4)$-matrix with $1$ in the $(i,j)$-th entry and zeros elsewhere.
We take iterative mutations $\mu_{2n+1}\,\mu_{4n+2}\,\mu_{2n-1}\ldots\mu_{2n+6}\,\mu_3\,\mu_{2n+4}\,\mu_1$.
In each step the coefficients and the skew-symmetric matrix are transformed as
\[\begin{array}{rcl}
	\left(y_1,\ldots,y_{4n+4},\left(\lambda_{i,j}\right)_{i,j}\right) & \xrightarrow{\mu_1} & \left(y^{(1)}_1,\ldots,y^{(1)}_{4n+4},\left(\lambda^{(1)}_{i,j}\right)_{i,j}\right) \\[4pt]
	& \xrightarrow{\mu_{2n+4}} & \left(y^{(2n+4)}_1,\ldots,y^{(2n+4)}_{4n+4},\left(\lambda^{(2n+4)}_{i,j}\right)_{i,j}\right) \\[4pt]
	& \vdots \\[4pt]
	& \xrightarrow{\mu_{2n+1}} & \left(y^{(2n+1)}_1,\ldots,y^{(2n+1)}_{4n+4},\left(\lambda^{(2n+1)}_{i,j}\right)_{i,j}\right).
\end{array}\]
The matrix $\left(\lambda^{(1)}_{i,j}\right)_{i,j}$ is given by
\begin{align*}
	&\lambda^{(1)}_{1,2n+3} = 1,\quad
	\lambda^{(1)}_{1,2n+4} = -1,\quad
	\lambda^{(1)}_{1,4n+3} = -1,\quad
	\lambda^{(1)}_{1,4n+4} = 1, \\
	&\lambda^{(1)}_{2n+3,2n+4} = 1,\quad
	\lambda^{(1)}_{2n+3,4n+3} = 1,\quad
	\lambda^{(1)}_{2n+4,4n+4} = -1,\quad
	\lambda^{(1)}_{4n+3,4n+4} = -1,
\end{align*}
where $\lambda^{(1)}_{1,2n+3}=1$ stands for 
\[
	\lambda^{(1)}_{1,2n+3} = 1,\quad \lambda^{(1)}_{2n+3,1} = -1.
\]
Here we write only entries of the matrix which are changed by the mutation $\mu_1$.
In a similar manner, the matrices $\left(\lambda^{(k)}_{i,j}\right)_{i,j}$ are given by
\begin{align*}
	&\lambda^{(2n+4)}_{2n+4,1} = -1,\quad
	\lambda^{(2n+4)}_{2n+4,2} = -1,\quad
	\lambda^{(2n+4)}_{2n+4,3} = -1,\quad
	\lambda^{(2n+4)}_{2n+4,4} = 1,\quad
	\lambda^{(2n+4)}_{2n+4,2n+3} = 1,\quad
	\lambda^{(2n+4)}_{2n+4,4n+4} = 1, \\
	&\lambda^{(2n+4)}_{j,4} = -1,\quad
	\lambda^{(2n+4)}_{j,2n+3} = 0\quad (j=1,2,3),\quad
	\lambda^{(2n+4)}_{1,4n+4} = 0,\quad
	\lambda^{(2n+4)}_{2,4n+4} = 0,\quad
	\lambda^{(2n+4)}_{3,4n+4} = -1,
\end{align*}
for $k=2n+4$,
\begin{align*}
	&\lambda^{(2i+1)}_{2i+1,2i+2} = 1,\quad
	\lambda^{(2i+1)}_{2i+1,2i+2n+2} = -1,\quad
	\lambda^{(2i+1)}_{2i+1,2i+2n+3} = 1,\quad
	\lambda^{(2i+1)}_{2i+1,2i+2n+4} = -1,\quad
	\lambda^{(2i+1)}_{2i+1,4n+4} = 1, \\
	&\lambda^{(2i+1)}_{j,2i+2} = 0,\quad
	\lambda^{(2i+1)}_{j,2i+2n+3} = -1\quad (j=2i+2n+2,2i+2n+4), \\
	&\lambda^{(2i+1)}_{2i+2n+2,4n+4} = 0,\quad
	\lambda^{(2i+1)}_{2i+2n+4,4n+4} = -1,
\end{align*}
for $k=2i+1$ with $i=1,\ldots,n-1$,
\begin{align*}
	&\lambda^{(2i+2n+4)}_{2i+2n+4,2i+1} = -1,\quad
	\lambda^{(2i+2n+4)}_{2i+2n+4,2i+3} = -1,\quad
	\lambda^{(2i+2n+4)}_{2i+2n+4,2i+4} = 1,\quad
	\lambda^{(2i+2n+4)}_{2i+2n+4,2i+2n+3} = 1, \\
	&\lambda^{(2i+2n+4)}_{2i+2n+4,4n+4} = 1, \\
	&\lambda^{(2i+2n+4)}_{j,2i+4} = -1,\quad
	\lambda^{(2i+2n+4)}_{j,2i+2n+3} = 0\quad (j=2i+1,2i+3),\quad
	\lambda^{(2i+2n+4)}_{2i+1,4n+4} = 0,\quad
	\lambda^{(2i+2n+4)}_{2i+3,4n+4} = \delta_{i,n-1} - 1,
\end{align*}
for $k=2i+2n+4$ with $i=1,\ldots,n-1$ and
\[
	\lambda^{(2n+1)}_{2n+1,2n+2} = 1,\quad
	\lambda^{(2n+1)}_{2n+1,4n+2} = -1,\quad
	\lambda^{(2n+1)}_{2n+1,4n+3} = 1,\quad
	\lambda^{(2n+1)}_{2n+2,4n+2} = 0,\quad
	\lambda^{(2n+1)}_{4n+2,4n+3} = -1,
\]
for $k=2n+1$.
The coefficients $y^{(1)}_1,\ldots,y^{(1)}_{4n+4}$ are given by
\[
	y^{(1)}_1 = \frac{1}{y_1},\quad
	y^{(1)}_{2n+3} = y_{2n+3}\,(1+y_1),\quad
	y^{(1)}_{2n+4} = \frac{y_{2n+4}}{1+\frac{1}{y_1}},\quad
	y^{(1)}_{4n+3} = \frac{y_{4n+3}}{1+\frac{1}{y_1}},\quad
	y^{(1)}_{4n+4} = y_{4n+4}\,(1+y_1).
\]
Here we write only coefficients which are changed by the mutation $\mu_1$.
In a similar manner, the coefficients $y^{(k)}_1,\ldots,y^{(k)}_{4n+4}$ are given by
\begin{align*}
	&y^{(2n+4)}_1 = \frac{y^{(1)}_1}{1+\frac{1}{y^{(1)}_{2n+4}}},\quad
	y^{(2n+4)}_2 = \frac{y^{(1)}_2}{1+\frac{1}{y^{(1)}_{2n+4}}},\quad
	y^{(2n+4)}_3 = \frac{y^{(1)}_3}{1+\frac{1}{y^{(1)}_{2n+4}}},\quad
	y^{(2n+4)}_4 = y^{(1)}_4\,(1+y^{(1)}_{2n+4}), \\
	&y^{(2n+4)}_{2n+3} = y^{(1)}_{2n+3}\,(1+y^{(1)}_{2n+4}),\quad
	y^{(2n+4)}_{2n+4} = \frac{1}{y^{(1)}_{2n+4}},\quad
	y^{(2n+4)}_{4n+4} = y^{(1)}_{4n+4}\,(1+y^{(1)}_{2n+4}),
\end{align*}
for $k=2n+4$,
\begin{align*}
	&y^{(2i+1)}_{2i+1} = \frac{1}{y^{(2i+2n+2)}_{2i+1}},\quad
	y^{(2i+1)}_{2i+2} = y^{(2i+2n+2)}_{2i+2}\,(1+y^{(2i+2n+2)}_{2i+1}),\quad
	y^{(2i+1)}_{2i+2n+2} = \frac{y^{(2i+2n+2)}_{2i+2n+2}}{1+\frac{1}{y^{(2i+2n+2)}_{2i+1}}}, \\
	&y^{(2i+1)}_{2i+2n+3} = y^{(2i+2n+2)}_{2i+2n+3}\,(1+y^{(2i+2n+2)}_{2i+1}),\quad
	y^{(2i+1)}_{2i+2n+4} = \frac{y^{(2i+2n+2)}_{2i+2n+4}}{1+\frac{1}{y^{(2i+2n+2)}_{2i+1}}},\quad
	y^{(2i+1)}_{4n+4} = y^{(2i+2n+2)}_{4n+4}\,(1+y^{(2i+2n+2)}_{2i+1}),
\end{align*}
for $k=2i+1$ with $i=1,\ldots,n-1$,
\begin{align*}
	&y^{(2i+2n+4)}_{2i+1} = \frac{y^{(2i+1)}_{2i+1}}{1+\frac{1}{y^{(2i+1)}_{2i+2n+4}}},\quad
	y^{(2i+2n+4)}_{2i+3} = \frac{y^{(2i+1)}_{2i+3}}{1+\frac{1}{y^{(2i+1)}_{2i+2n+4}}},\quad
	y^{(2i+2n+4)}_{2i+4} = y^{(2i+1)}_{2i+4}\,(1+y^{(2i+1)}_{2i+2n+4}), \\
	&y^{(2i+2n+4)}_{2i+2n+3} = y^{(2i+1)}_{2i+2n+3}\,(1+y^{(2i+1)}_{2i+2n+4}),\quad
	y^{(2i+2n+4)}_{2i+2n+4} = \frac{1}{y^{(2i+1)}_{2i+2n+4}},\quad
	y^{(2i+2n+4)}_{4n+4} = y^{(2i+1)}_{4n+4}\,(1+y^{(2i+1)}_{2i+2n+4}),
\end{align*}
for $k=2i+2n+4$ with $i=1,\ldots,n-1$ and
\begin{align*}
	&y^{(2n+1)}_{2n+1} = \frac{1}{y^{(4n+2)}_{2n+1}},\quad
	y^{(2n+1)}_{2n+2} = y^{(4n+2)}_{2n+2}\,(1+y^{(4n+2)}_{2n+1}),\quad
	y^{(2n+1)}_{4n+2} = \frac{y^{(4n+2)}_{4n+2}}{1+\frac{1}{y^{(4n+2)}_{2n+1}}}, \\
	&y^{(2n+1)}_{4n+3} = y^{(4n+2)}_{4n+3}\,(1+y^{(4n+2)}_{2n+1}),
\end{align*}
for $k=2n+1$.

After taking iterative mutations $\mu_{2n+1}\ldots\mu_1$ and a permutation $(2n+1,4n+4)$, we have
\[
	\left(y_1,\ldots,y_{4n+4},\left(\lambda_{i,j}\right)_{i,j}\right) \xrightarrow{(2n+1,4n+4)\,\mu_{2n+1}\ldots\mu_1} \left(\hat{y}_1,\ldots,\hat{y}_{4n+4},\left(\hat{\lambda}_{i,j}\right)_{i,j}\right),
\]
where
\begin{align*}
	&\hat{y}_1 = \frac{y_{2n+4}}{1+y_1+y_1\,y_{2n+4}},\quad
	\hat{y}_2 = \frac{y_1\,y_2\,y_{2n+4}}{1+y_1+y_1\,y_{2n+4}}, \\
	&\hat{y}_{2i+1} = y_{2i+2n+4}\,\frac{\sum_{j=0}^{i-1}\left(\prod_{k=0}^{j-1}y_{2k+1}\,y_{2k+2n+4}\right)(1+y_{2j+1})+\prod_{k=0}^{i-1}y_{2k+1}\,y_{2k+2n+4}}{\sum_{j=0}^{i}\left(\prod_{k=0}^{j-1}y_{2k+1}\,y_{2k+2n+4}\right)(1+y_{2j+1})+\prod_{k=0}^{i}y_{2k+1}\,y_{2k+2n+4}}\quad (i=1,\ldots,n-1), \\
	&\hat{y}_{2i+2} = y_{2i+2}\,\frac{\sum_{j=0}^{i}\left(\prod_{k=0}^{j-1}y_{2k+1}\,y_{2k+2n+4}\right)(1+y_{2j+1})}{\sum_{j=0}^{i-1}\left(\prod_{k=0}^{j-1}y_{2k+1}\,y_{2k+2n+4}\right)(1+y_{2j+1})}\quad (i=1,\ldots,n), \\
	&\hat{y}_{2n+1} = y_{4n+4}\sum_{j=0}^{n-1}\left(\prod_{k=0}^{j-1}y_{2k+1}\,y_{2k+2n+4}\right)(1+y_{2j+1}) + \prod_{k=0}^{n-1}y_{2k+1}\,y_{2k+2n+4}, \\
	&\hat{y}_{2i+2n+1} = y_{2i+2n+1}\,\frac{\sum_{j=0}^{i-1}\left(\prod_{k=0}^{j-1}y_{2k+1}\,y_{2k+2n+4}\right)(1+y_{2j+1})+\prod_{k=0}^{i-1}y_{2k+1}\,y_{2k+2n+4}}{\sum_{j=0}^{i-2}\left(\prod_{k=0}^{j-1}y_{2k+1}\,y_{2k+2n+4}\right)(1+y_{2j+1})+\prod_{k=0}^{i-2}y_{2k+1}\,y_{2k+2n+4}}\quad (i=1,\ldots,n), \\
	&\hat{y}_{2i+2n+2} = y_{2i+1}\,\frac{\sum_{j=0}^{i-1}\left(\prod_{k=0}^{j-1}y_{2k+1}\,y_{2k+2n+4}\right)(1+y_{2j+1})}{\sum_{j=0}^{i}\left(\prod_{k=0}^{j-1}y_{2k+1}\,y_{2k+2n+4}\right)(1+y_{2j+1})}\quad (i=1,\ldots,n), \\
	&\hat{y}_{4n+3} = \frac{y_1\,y_{4n+3}}{1+y_1}\frac{\sum_{j=0}^{n}\left(\prod_{k=0}^{j-1}y_{2k+1}\,y_{2k+2n+4}\right)(1+y_{2j+1})}{\sum_{j=0}^{n-1}\left(\prod_{k=0}^{j-1}y_{2k+1}\,y_{2k+2n+4}\right)(1+y_{2j+1})+\prod_{k=0}^{n-1}y_{2k+1}\,y_{2k+2n+4}}, \\
	&\hat{y}_{4n+4} = \frac{\sum_{j=0}^{n-1}\left(\prod_{k=0}^{j-1}y_{2k+1}\,y_{2k+2n+4}\right)(1+y_{2j+1})+\prod_{k=0}^{n-1}y_{2k+1}\,y_{2k+2n+4}}{\left(\prod_{k=0}^{n-1}y_{2k+1}\,y_{2k+2n+4}\right)y_{2n+1}}, \\
\end{align*}
and
\begin{align*}
	\left(\hat{\lambda}_{i,j}\right)_{i,j}
	&= -X_{1,4} + X_{1,2n+4} - X_{1,4n+3} - X_{2,4} + X_{2,2n+4} - X_{2,4n+3} \\
	&\quad + \sum_{i=1}^{n-1}(X_{2i+1,2i+2}-X_{2i+1,2i+6}-X_{2i+1,2i+2n+2}+X_{2i+1,2i+2n+4}-X_{2i+2,2i+2n+1}+X_{2i+2,2i+2n+3}) \\
	&\quad + X_{2n+1,2n+2} - X_{2n+1,4n+2} + X_{2n+1,4n+3} - X_{2n+2,4n+1} + X_{2n+2,4n+3} - X_{2n+2,4n+4} \\
	&\quad - X_{2n+3,2n+4} + X_{2n+3,4n+3} - \sum_{i=1}^{n-1}(X_{2i+2n+2,2i+2n+3}+X_{2i+2n+3,2i+2n+4}) \\
	&\quad - X_{4n+2,4n+3} + X_{4n+2,4n+4} - X_{4n+3,4n+4}.
\end{align*}

We next take iterative mutations $\mu_1\,\mu_{2n+4}\,\mu_3\,\mu_{2n+6}\ldots\mu_{2n-1}\,\mu_{4n+2}\,\mu_{2n+1}$.
In each step the coefficients and the skew-symmetric matrix are transformed as
\[\begin{array}{rcl}
	\left(\hat{y}_1,\ldots,\hat{y}_{4n+4},\left(\hat{\lambda}_{i,j}\right)_{i,j}\right) & \xrightarrow{\mu_{2n+1}} & \left(\hat{y}^{(2n+1)}_1,\ldots,\hat{y}^{(2n+1)}_{4n+4},\left(\hat{\lambda}^{(2n+1)}_{i,j}\right)_{i,j}\right) \\[4pt]
	& \xrightarrow{\mu_{4n+2}} & \left(\hat{y}^{(4n+2)}_1,\ldots,\hat{y}^{(4n+2)}_{4n+4},\left(\hat{\lambda}^{(4n+2)}_{i,j}\right)_{i,j}\right) \\[4pt]
	& \vdots \\[4pt]
	& \xrightarrow{\mu_1} & \left(\hat{y}^{(1)}_1,\ldots,\hat{y}^{(1)}_{4n+4},\left(\hat{\lambda}^{(1)}_{i,j}\right)_{i,j}\right).
\end{array}\]
In a similar manner as above, the matrices $\left(\lambda^{(k)}_{i,j}\right)_{i,j}$ are given by
\[
	\hat{\lambda}^{(2n+1)}_{2n+1,2n+2} = -1,\quad
	\hat{\lambda}^{(2n+1)}_{2n+1,4n+2} = 1,\quad
	\hat{\lambda}^{(2n+1)}_{2n+1,4n+3} = -1,\quad
	\hat{\lambda}^{(2n+1)}_{2n+2,4n+2} = -1,\quad
	\hat{\lambda}^{(2n+1)}_{4n+2,4n+3} = 0,
\]
for $k=2n+1$,
\begin{align*}
	&\hat{\lambda}^{(2i+2n+4)}_{2i+2n+4,2i+1} = 1,\quad
	\hat{\lambda}^{(2i+2n+4)}_{2i+2n+4,2i+3} = 1,\quad
	\hat{\lambda}^{(2i+2n+4)}_{2i+2n+4,2i+4} = -1,\quad
	\hat{\lambda}^{(2i+2n+4)}_{2i+2n+4,2i+2n+3} = -1, \\
	&\hat{\lambda}^{(2i+2n+4)}_{2i+2n+4,4n+4} = -1, \\
	&\hat{\lambda}^{(2i+2n+4)}_{j,2i+4} = 0,\quad
	\hat{\lambda}^{(2i+2n+4)}_{j,2i+2n+3} = 1\quad (j=2i+1,2i+3),\quad
	\hat{\lambda}^{(2i+2n+4)}_{2i+1,4n+4} = 1,\quad
	\hat{\lambda}^{(2i+2n+4)}_{2i+3,4n+4} = \delta_{i,n-1},
\end{align*}
for $k=2i+2n+4$ with $i=n-1,\ldots,1$,
\begin{align*}
	&\hat{\lambda}^{(2i+1)}_{2i+1,2i+2} = -1,\quad
	\hat{\lambda}^{(2i+1)}_{2i+1,2i+2n+2} = 1,\quad
	\hat{\lambda}^{(2i+1)}_{2i+1,2i+2n+3} = -1,\quad
	\hat{\lambda}^{(2i+1)}_{2i+1,2i+2n+4} = 1,\quad
	\hat{\lambda}^{(2i+1)}_{2i+1,4n+4} = -1, \\
	&\hat{\lambda}^{(2i+1)}_{j,2i+2} = 1,\quad
	\hat{\lambda}^{(2i+1)}_{j,2i+2n+3} = 0\quad (j=2i+2n+2,2i+2n+4), \\
	&\hat{\lambda}^{(2i+1)}_{2i+2n+2,4n+4} = 1,\quad
	\hat{\lambda}^{(2i+1)}_{2i+2n+4,4n+4} = 0,
\end{align*}
for $k=2i+1$ with $i=n-1,\ldots,1$,
\begin{align*}
	&\hat{\lambda}^{(2n+4)}_{2n+4,1} = 1,\quad
	\hat{\lambda}^{(2n+4)}_{2n+4,2} = 1,\quad
	\hat{\lambda}^{(2n+4)}_{2n+4,3} = 1,\quad
	\hat{\lambda}^{(2n+4)}_{2n+4,4} = -1,\quad
	\hat{\lambda}^{(2n+4)}_{2n+4,2n+3} = -1,\quad
	\hat{\lambda}^{(2n+4)}_{2n+4,4n+4} = -1, \\
	&\hat{\lambda}^{(2n+4)}_{j,4} = 0,\quad
	\hat{\lambda}^{(2n+4)}_{j,2n+3} = 1\quad (j=1,2,3),\quad
	\hat{\lambda}^{(2n+4)}_{1,4n+4} = 1,\quad
	\hat{\lambda}^{(2n+4)}_{2,4n+4} = 1,\quad
	\hat{\lambda}^{(2n+4)}_{3,4n+4} = 0,
\end{align*}
for $k=2n+4$ and
\begin{align*}
	&\hat{\lambda}^{(1)}_{1,2n+3} = -1,\quad
	\hat{\lambda}^{(1)}_{1,2n+4} = 1,\quad
	\hat{\lambda}^{(1)}_{1,4n+3} = 1,\quad
	\hat{\lambda}^{(1)}_{1,4n+4} = -1, \\
	&\hat{\lambda}^{(1)}_{2n+3,2n+4} = 0,\quad
	\hat{\lambda}^{(1)}_{2n+3,4n+3} = 0,\quad
	\hat{\lambda}^{(1)}_{2n+4,4n+4} = 0,\quad
	\hat{\lambda}^{(1)}_{4n+3,4n+4} = 0,
\end{align*}
for $k=1$.
The coefficients $\hat{y}^{(k)}_1,\ldots,\hat{y}^{(k)}_{4n+4}$ are given by
\[
	\hat{y}^{(2n+1)}_{2n+1} = \frac{1}{\hat{y}_{2n+1}},\quad
	\hat{y}^{(2n+1)}_{2n+2} = \frac{\hat{y}_{2n+2}}{1+\frac{1}{\hat{y}_{2n+1}}},\quad
	\hat{y}^{(2n+1)}_{4n+2} = \hat{y}_{4n+2}\,(1+\hat{y}_{2n+1}),\quad
	\hat{y}^{(2n+1)}_{4n+3} = \frac{\hat{y}_{4n+3}}{1+\frac{1}{\hat{y}_{2n+1}}},
\]
for $k=2n+1$,
\begin{align*}
	&\hat{y}^{(2i+2n+4)}_{2i+1} = \hat{y}^{(2i+3)}_{2i+1}\,(1+\hat{y}^{(2i+3)}_{2i+2n+4}),\quad
	\hat{y}^{(2i+2n+4)}_{2i+3} = \hat{y}^{(2i+3)}_{2i+3}\,(1+\hat{y}^{(2i+3)}_{2i+2n+4}),\quad
	\hat{y}^{(2i+2n+4)}_{2i+4} = \frac{\hat{y}^{(2i+3)}_{2i+4}}{1+\frac{1}{\hat{y}^{(2i+3)}_{2i+2n+4}}}, \\
	&\hat{y}^{(2i+2n+4)}_{2i+2n+3} = \frac{\hat{y}^{(2i+3)}_{2i+2n+3}}{1+\frac{1}{\hat{y}^{(2i+3)}_{2i+2n+4}}},\quad
	\hat{y}^{(2i+2n+4)}_{2i+2n+4} = \frac{1}{\hat{y}^{(2i+3)}_{2i+2n+4}},\quad
	\hat{y}^{(2i+2n+4)}_{4n+4} = \frac{\hat{y}^{(2i+3)}_{4n+4}}{1+\frac{1}{\hat{y}^{(2i+3)}_{2i+2n+4}}},
\end{align*}
for $k=2i+2n+4$ with $i=n-1,\ldots,1$,
\begin{align*}
	&\hat{y}^{(2i+1)}_{2i+1} = \frac{1}{\hat{y}^{(2i+2n+4)}_{2i+1}},\quad
	\hat{y}^{(2i+1)}_{2i+2} = \frac{\hat{y}^{(2i+2n+4)}_{2i+2}}{1+\frac{1}{\hat{y}^{(2i+2n+4)}_{2i+1}}},\quad
	\hat{y}^{(2i+1)}_{2i+2n+2} = \hat{y}^{(2i+2n+4)}_{2i+2n+2}\,(1+\hat{y}^{(2i+2n+4)}_{2i+1}), \\
	&\hat{y}^{(2i+1)}_{2i+2n+3} = \frac{\hat{y}^{(2i+2n+4)}_{2i+2n+3}}{1+\frac{1}{\hat{y}^{(2i+2n+4)}_{2i+1}}},\quad
	\hat{y}^{(2i+1)}_{2i+2n+4} = \hat{y}^{(2i+2n+4)}_{2i+2n+4}\,(1+\hat{y}^{(2i+2n+4)}_{2i+1}),\quad
	\hat{y}^{(2i+1)}_{4n+4} = \frac{\hat{y}^{(2i+2n+4)}_{4n+4}}{1+\frac{1}{\hat{y}^{(2i+2n+4)}_{2i+1}}},
\end{align*}
for $k=2i+1$ with $i=n-1,\ldots,1$,
\begin{align*}
	&\hat{y}^{(2n+4)}_1 = \hat{y}^{(3)}_1\,(1+\hat{y}^{(3)}_{2n+4}),\quad
	\hat{y}^{(2n+4)}_2 = \hat{y}^{(3)}_2\,(1+\hat{y}^{(3)}_{2n+4}),\quad
	\hat{y}^{(2n+4)}_3 = \hat{y}^{(3)}_3\,(1+\hat{y}^{(3)}_{2n+4}), \\
	&\hat{y}^{(2n+4)}_4 = \frac{\hat{y}^{(3)}_4}{1+\frac{1}{\hat{y}^{(3)}_{2n+4}}},\quad
	\hat{y}^{(2n+4)}_{2n+3} = \frac{\hat{y}^{(3)}_{2n+3}}{1+\frac{1}{\hat{y}^{(3)}_{2n+4}}},\quad
	\hat{y}^{(2n+4)}_{2n+4} = \frac{1}{\hat{y}^{(3)}_{2n+4}},\quad
	\hat{y}^{(2n+4)}_{4n+4} = \frac{\hat{y}^{(3)}_{4n+4}}{1+\frac{1}{\hat{y}^{(3)}_{2n+4}}},
\end{align*}
for $k=2n+4$ and
\begin{align*}
	&\hat{y}^{(1)}_1 = \frac{1}{\hat{y}^{(2n+4)}_1},\quad
	\hat{y}^{(1)}_{2n+3} = \frac{\hat{y}^{(2n+4)}_{2n+3}}{1+\frac{1}{\hat{y}^{(2n+4)}_1}},\quad
	\hat{y}^{(1)}_{2n+4} = \hat{y}^{(2n+4)}_{2n+4}\,(1+\hat{y}^{(2n+4)}_1),\quad
	\hat{y}^{(1)}_{4n+3} = \hat{y}^{(2n+4)}_{4n+3}\,(1+\hat{y}^{(2n+4)}_1), \\
	&\hat{y}^{(1)}_{4n+4} = \frac{\hat{y}^{(2n+4)}_{4n+4}}{1+\frac{1}{\hat{y}^{(2n+4)}_1}},
\end{align*}
for $k=1$.

After taking iterative mutations $\mu_1\ldots\mu_{2n+1}$, the matrix $\left(\hat{\lambda}^{(1)}_{i,j}\right)_{i,j}$ is equivalent to $\Lambda$ defined by \eqref{Eq:Quiver_Matrix}.
Moreover, we have
\begin{equation}\begin{split}\label{Proof:s0_Transf}
	\hat{y}^{(1)}_{2i+1} &= \frac{\sum_{j=0}^{n}\left(\prod_{k=0}^{j-1}y_{2i+2k+1}\,y_{2i+2k+2n+4}\right)\left(1+y_{2i+2j+1}\right)}{y_{2i+2n+4}\sum_{j=0}^{n}\left(\prod_{k=0}^{j-1}y_{2i+2k+3}\,y_{2i+2k+2n+6}\right)\left(1+y_{2i+2j+3}\right)}, \\
	\hat{y}^{(1)}_{2i+1}\,\hat{y}^{(1)}_{2i+2} &= y_{2i+1}\,y_{2i+2}, \\
	\hat{y}^{(1)}_{2i+2n+4} &= \frac{\sum_{j=0}^{n}\left(\prod_{k=0}^{j-1}y_{2i+2k+2n+4}\,y_{2i+2k+3}\right)\left(1+y_{2i+2j+2n+4}\right)}{y_{2i+3}\sum_{j=0}^{n}\left(\prod_{k=0}^{j-1}y_{2i+2k+2n+6}\,y_{2i+2k+5}\right)\left(1+y_{2i+2j+2n+6}\right)}, \\
	\hat{y}^{(1)}_{2i+2n+3}\,\hat{y}^{(1)}_{2i+2n+4} &= y_{2i+2n+3}\,y_{2i+2n+4},
\end{split}\end{equation}
for $i=0,\ldots,n$.
Here we assume that the indices of $y_i$ are congruent modulo $4n+4$.
Then, substituting
\[
	y_{2i+1} = \varphi_{2i},\quad
	y_{2i+2} = \frac{\alpha_{2i}}{\varphi_{2i}},\quad
	y_{2i+2n+3} = \varphi_{2i+1},\quad
	y_{2i+2n+4} = \frac{\alpha_{2i+1}}{\varphi_{2i+1}}\quad (i=0,\ldots,n),
\]
to \eqref{Proof:s0_Transf}, we obtain the action of $s_0$ given by \eqref{Eq:s0s1_Transf_1} and \eqref{Eq:s0s1_Transf_2}.

\section{Proof of Theorem \ref{Thm:Fund_Rel}}\label{App:Prf_Fund_Rel}

We have to verify that the fundamental relations
\begin{align}
	&r_i^2 = 1\quad (i\in\mathbb{Z}_{2n+2}), \label{Eq:Fund_Rel_1} \\
	&r_i\,r_j\,r_i = r_j\,r_i\,r_j\quad (i,j\in\mathbb{Z}_{2n+2},\ a_{i,j}=-1), \label{Eq:Fund_Rel_2} \\
	&r_i\,r_j = r_j\,r_i\quad (i,j\in\mathbb{Z}_{2n+2},\ a_{i,j}=0), \label{Eq:Fund_Rel_3} \\
	&s_k^2 = 1,\quad
	(s'_k)^2 = 1\quad (k\in\mathbb{Z}_2), \label{Eq:Fund_Rel_4}
\end{align}
and the commutative relations
\begin{align}
	&r_i\,s_k = s_k\,r_i,\quad
	r_i\,s'_k = s'_k\,r_i\quad (i\in\mathbb{Z}_{2n+2},\ k\in\mathbb{Z}_2), \label{Eq:Fund_Rel_5} \\
	&s_k\,s'_l = s'_l\,s_k\quad (k,l\in\mathbb{Z}_2), \label{Eq:Fund_Rel_6}
\end{align}
are satisfied.
Among them, relations \eqref{Eq:Fund_Rel_1}, \eqref{Eq:Fund_Rel_2} and \eqref{Eq:Fund_Rel_3} have been already shown in \cite{KNY1,KNY2}.
Moreover, we can show relations \eqref{Eq:Fund_Rel_4} and \eqref{Eq:Fund_Rel_5} by direct calculations with relation \eqref{Eq:Fund_Rel_Quiver}.
Hence the remaining problem is relation \eqref{Eq:Fund_Rel_6}.
We show only the relation $s_1\,s'_1=s'_1\,s_1$ because the other relations can be obtained from $s_0=\pi^{-1}s_1\,\pi$ and $s'_0=(\pi')^{-1}s'_1\,\pi'$.
Since the actions of $s_1\,s'_1$ and $s'_1\,s_1$ on the parameters are obvious, we investigate their actions on the dependent variables.

Recall that the simple reflections $s_1,s'_1$ act on the dependent variables as
\begin{align*}
	&s_1(\varphi_{2i}) = \alpha_{2i}\,\varphi_{2i+1}\frac{S_{2i+2}}{S_{2i}},\quad
	s_1(\varphi_{2i+1}) = \frac{\varphi_{2i+2}}{\alpha_{2i+2}}\frac{S_{2i+1}}{S_{2i+3}}, \\
	&s'_1(\varphi_{2i}) = \frac{\alpha_{2i}\,\alpha_{2i+1}}{\varphi_{2i+1}}\frac{S'_{2i}}{S'_{2i+2}},\quad
	s'_1(\varphi_{2i+1}) = \frac{\alpha_{2i+1}\,\alpha_{2i+2}}{\varphi_{2i+2}}\frac{S'_{2i+1}}{S'_{2i+3}},
\end{align*}
for $i=0,\ldots,n$, where
\begin{align*}
	&S_{2i} = \sum_{j=0}^{n}\left(\prod_{k=0}^{j-1}\frac{\alpha_{2i+2k}}{\varphi_{2i+2k}}\varphi_{2i+2k+1}\right)\left(1+\frac{\alpha_{2i+2j}}{\varphi_{2i+2j}}\right),\quad
	S_{2i+1} = \sum_{j=0}^{n}\left(\prod_{k=0}^{j-1}\varphi_{2i+2k+1}\frac{\alpha_{2i+2k+2}}{\varphi_{2i+2k+2}}\right)(1+\varphi_{2i+2j+1}), \\
	&S'_{2i} = \sum_{j=0}^{n}\left(\prod_{k=0}^{j-1}\frac{\varphi_{2i+2k}}{\alpha_{2i+2k}}\frac{\varphi_{2i+2k+1}}{\alpha_{2i+2k+1}}\right)\left(1+\frac{\varphi_{2i+2j}}{\alpha_{2i+2j}}\right),\quad
	S'_{2i+1} = \sum_{j=0}^{n}\left(\prod_{k=0}^{j-1}\frac{\varphi_{2i+2k+1}}{\alpha_{2i+2k+1}}\frac{\varphi_{2i+2k+2}}{\alpha_{2i+2k+2}}\right)\left(1+\frac{\varphi_{2i+2j+1}}{\alpha_{2i+2j+1}}\right),
\end{align*}
for $i=0,\ldots,n$.
We assume that the indices of $S_i,S'_i$ are congruent modulo $2n+2$ as well as $\alpha_i,\varphi_i$ are.

\begin{lem}
Equations
\begin{align}
	S_{2i} + \varphi_{2i+1}\,S_{2i+2} &= \left(1+\frac{\alpha_{2i}}{\varphi_{2i}}\right)S_{2i+1}, \label{Eq:Action_s'1s1_Lem_1} \\
	S'_{2i} + \frac{\varphi_{2i+1}}{\alpha_{2i+1}}\,S'_{2i+2} &= \left(1+\frac{\varphi_{2i}}{\alpha_{2i}}\right)S'_{2i+1}, \label{Eq:Action_s'1s1_Lem_2} \\
	S_{2i} - \frac{\alpha_{2i}}{\varphi_{2i}}\,S_{2i+1} &= 1 - \frac{q}{\beta_0}, \label{Eq:Action_s'1s1_Lem_3} \\
	S'_{2i} - \frac{\varphi_{2i}}{\alpha_{2i}}\,S'_{2i+1} &= 1 - \frac{\beta'_0}{q}, \label{Eq:Action_s'1s1_Lem_4}
\end{align}
are satisfied for $i=0,\ldots,n$.
\end{lem}

\begin{proof}
We show equation \eqref{Eq:Action_s'1s1_Lem_1}.
For each $i=0,\ldots,n$, the definition of $S_{2i}$ implies
\begin{align*}
	S_{2i} + \varphi_{2i+1}\,S_{2i+2} &= \sum_{j=0}^{n}\left(\prod_{k=0}^{j-1}\frac{\alpha_{2i+2k}}{\varphi_{2i+2k}}\varphi_{2i+2k+1}\right) + \sum_{j=0}^{n}\left(\prod_{k=0}^{j-1}\frac{\alpha_{2i+2k}}{\varphi_{2i+2k}}\varphi_{2i+2k+1}\right)\frac{\alpha_{2i+2j}}{\varphi_{2i+2j}} \\
	&\quad + \sum_{j=0}^{n}\varphi_{2i+1}\left(\prod_{k=0}^{j-1}\frac{\alpha_{2i+2k+2}}{\varphi_{2i+2k+2}}\varphi_{2i+2k+3}\right) + \sum_{j=0}^{n}\varphi_{2i+1}\left(\prod_{k=0}^{j-1}\frac{\alpha_{2i+2k+2}}{\varphi_{2i+2k+2}}\varphi_{2i+2k+3}\right)\frac{\alpha_{2i+2j+2}}{\varphi_{2i+2j+2}}.
\end{align*}
The second term of right-hand side is rewritten as
\[
	\sum_{j=0}^{n}\left(\prod_{k=0}^{j-1}\frac{\alpha_{2i+2k}}{\varphi_{2i+2k}}\varphi_{2i+2k+1}\right)\frac{\alpha_{2i+2j}}{\varphi_{2i+2j}} = \sum_{j=0}^{n}\frac{\alpha_{2i}}{\varphi_{2i}}\left(\prod_{k=0}^{j-1}\varphi_{2i+2k+1}\frac{\alpha_{2i+2k+2}}{\varphi_{2i+2k+2}}\right).
\]
The third term is rewritten as
\[
	\sum_{j=0}^{n}\varphi_{2i+1}\left(\prod_{k=0}^{j-1}\frac{\alpha_{2i+2k+2}}{\varphi_{2i+2k+2}}\varphi_{2i+2k+3}\right) = \sum_{j=0}^{n}\left(\prod_{k=0}^{j-1}\varphi_{2i+2k+1}\frac{\alpha_{2i+2k+2}}{\varphi_{2i+2k+2}}\right)\varphi_{2i+2j+1}.
\]
The first and fourth terms are rewritten as
\begin{align*}
	&\sum_{j=0}^{n}\left(\prod_{k=0}^{j-1}\frac{\alpha_{2i+2k}}{\varphi_{2i+2k}}\varphi_{2i+2k+1}\right) + \sum_{j=0}^{n}\varphi_{2i+1}\left(\prod_{k=0}^{j-1}\frac{\alpha_{2i+2k+2}}{\varphi_{2i+2k+2}}\varphi_{2i+2k+3}\right)\frac{\alpha_{2i+2j+2}}{\varphi_{2i+2j+2}} \\
	&= 1 + \frac{\alpha_{2i}}{\varphi_{2i}}\varphi_{2i+1} + \frac{\alpha_{2i}}{\varphi_{2i}}\varphi_{2i+1}\frac{\alpha_{2i+2}}{\varphi_{2i+2}}\varphi_{2i+3} + \ldots + \prod_{j=0}^{n-1}\frac{\alpha_{2i+2j}}{\varphi_{2i+2j}}\varphi_{2i+2j+1} \\
	&\quad + \varphi_{2i+1}\frac{\alpha_{2i+2}}{\varphi_{2i+2}} + \varphi_{2i+1}\frac{\alpha_{2i+2}}{\varphi_{2i+2}}\varphi_{2i+3}\frac{\alpha_{2i+4}}{\varphi_{2i+4}} + \ldots + \prod_{j=0}^{n-1}\varphi_{2i+2j+1}\frac{\alpha_{2i+2j+2}}{\varphi_{2i+2j+2}} + \prod_{j=0}^{n}\varphi_{2i+2j+1}\frac{\alpha_{2i+2j+2}}{\varphi_{2i+2j+2}} \\
	&= 1 + \varphi_{2i+1}\frac{\alpha_{2i+2}}{\varphi_{2i+2}} + \varphi_{2i+1}\frac{\alpha_{2i+2}}{\varphi_{2i+2}}\varphi_{2i+3}\frac{\alpha_{2i+4}}{\varphi_{2i+4}} + \ldots + \prod_{j=0}^{n-1}\varphi_{2i+2j+1}\frac{\alpha_{2i+2j+2}}{\varphi_{2i+2j+2}} \\
	&\quad + \frac{\alpha_{2i}}{\varphi_{2i}}\varphi_{2i+1} + \frac{\alpha_{2i}}{\varphi_{2i}}\varphi_{2i+1}\frac{\alpha_{2i+2}}{\varphi_{2i+2}}\varphi_{2i+3} + \ldots + \prod_{j=0}^{n-1}\frac{\alpha_{2i+2j}}{\varphi_{2i+2j}}\varphi_{2i+2j+1} + \prod_{j=0}^{n-1}\varphi_{2i+2j+1}\frac{\alpha_{2i+2j+2}}{\varphi_{2i+2j+2}} \\
	&= \sum_{j=0}^{n}\left(\prod_{k=0}^{j-1}\varphi_{2i+2k+1}\frac{\alpha_{2i+2k+2}}{\varphi_{2i+2k+2}}\right) + \sum_{j=0}^{n}\frac{\alpha_{2i}}{\varphi_{2i}}\left(\prod_{k=0}^{j-1}\varphi_{2i+2k+1}\frac{\alpha_{2i+2k+2}}{\varphi_{2i+2k+2}}\right)\varphi_{2i+2j+1}.
\end{align*}
It follows that
\[
	S_{2i} + \varphi_{2i+1}\,S_{2i+2} = \left(1+\frac{\alpha_{2i}}{\varphi_{2i}}\right)S_{2i+1}.
\]
The other equations can be shown in a similar manner.
\end{proof}

We first show $s'_1\,s_1(\varphi_{2i})=s_1\,s'_1(\varphi_{2i})$ by using this lemma.
For each $i=0,\ldots,n$, the action $s'_1\,s_1(\varphi_{2i})$ is described as
\begin{align*}
	s'_1\,s_1(\varphi_{2i}) &= \frac{\alpha_{2i}\,\alpha_{2i+1}}{s_1(\varphi_{2i+1})}\frac{\sum_{j=0}^{n}\left\{\prod_{k=0}^{j-1}\frac{s_1(\varphi_{2i+2k})}{\alpha_{2i+2k}}\frac{s_1(\varphi_{2i+2k+1})}{\alpha_{2i+2k+1}}\right\}\left\{1+\frac{s_1(\varphi_{2i+2j})}{\alpha_{2i+2j}}\right\}}{\sum_{j=0}^{n}\left\{\prod_{k=0}^{j-1}\frac{s_1(\varphi_{2i+2k+2})}{\alpha_{2i+2k+2}}\frac{s_1(\varphi_{2i+2k+3})}{\alpha_{2i+2k+3}}\right\}\left\{1+\frac{s_1(\varphi_{2i+2j+2})}{\alpha_{2i+2j+2}}\right\}} \\
	&= \frac{\alpha_{2i}\,\alpha_{2i+1}}{\frac{\varphi_{2i+2}}{\alpha_{2i+2}}}\frac{S_{2i+2}\sum_{j=0}^{n}\left(\prod_{k=0}^{j-1}\frac{\varphi_{2i+2k+1}}{\alpha_{2i+2k+1}}\frac{\varphi_{2i+2k+2}}{\alpha_{2i+2k+2}}\right)\frac{S_{2i+2j}+\varphi_{2i+2j+1}S_{2i+2j+2}}{S_{2i+2j+1}}}{S_{2i}\sum_{j=0}^{n}\left(\prod_{k=0}^{j-1}\frac{\varphi_{2i+2k+3}}{\alpha_{2i+2k+3}}\frac{\varphi_{2i+2k+4}}{\alpha_{2i+2k+4}}\right)\frac{S_{2i+2j+2}+\varphi_{2i+2j+3}S_{2i+2j+4}}{S_{2i+2j+3}}}.
\end{align*}
Then we have
\begin{align*}
	\sum_{j=0}^{n}\left(\prod_{k=0}^{j-1}\frac{\varphi_{2i+2k+1}}{\alpha_{2i+2k+1}}\frac{\varphi_{2i+2k+2}}{\alpha_{2i+2k+2}}\right)\frac{S_{2i+2j}+\varphi_{2i+2j+1}S_{2i+2j+2}}{S_{2i+2j+1}}
	&= \sum_{j=0}^{n}\left(\prod_{k=0}^{j-1}\frac{\varphi_{2i+2k+1}}{\alpha_{2i+2k+1}}\frac{\varphi_{2i+2k+2}}{\alpha_{2i+2k+2}}\right)\left(1+\frac{\alpha_{2i+2j}}{\varphi_{2i+2j}}\right) \\
	&= \frac{\alpha_{2i}}{\varphi_{2i}}\,S'_{2i},
\end{align*}
by using equation \eqref{Eq:Action_s'1s1_Lem_1}.
It follows that
\begin{equation}\label{Eq:Action_s'1s1_1}
	s'_1\,s_1(\varphi_{2i}) = \frac{\alpha_{2i}\,\alpha_{2i+1}}{\frac{\varphi_{2i}}{\alpha_{2i}}}\frac{S_{2i+2}\,S'_{2i}}{S_{2i}\,S'_{2i+2}}.
\end{equation}
In a similar manner, we obtain
\[
	s_1\,s'_1(\varphi_{2i}) = \frac{\alpha_{2i}\,\alpha_{2i+1}}{\frac{\varphi_{2i}}{\alpha_{2i}}}\frac{S'_{2i}\,S_{2i+2}}{S'_{2i+2}\,S_{2i}},
\]
by using equation \eqref{Eq:Action_s'1s1_Lem_2}.

We next show $s'_1\,s_1(\varphi_{2i+1})=s_1\,s'_1(\varphi_{2i+1})$.
For each $i=0,\ldots,n$, the action $s'_1\,s_1(\varphi_{2i+1})$ is described as
\begin{align*}
	s'_1\,s_1(\varphi_{2i+1}) &= \frac{\alpha_{2i+1}\,\alpha_{2i+2}}{s_1(\varphi_{2i+2})}\frac{\sum_{j=0}^{n}\left\{\prod_{k=0}^{j-1}\frac{s_1(\varphi_{2i+2k+1})}{\alpha_{2i+2k+1}}\frac{s_1(\varphi_{2i+2k+2})}{\alpha_{2i+2k+2}}\right\}\left\{1+\frac{s_1(\varphi_{2i+2j+1})}{\alpha_{2i+2j+1}}\right\}}{\sum_{j=0}^{n}\left\{\prod_{k=0}^{j-1}\frac{s_1(\varphi_{2i+2k+3})}{\alpha_{2i+2k+3}}\frac{s_1(\varphi_{2i+2k+4})}{\alpha_{2i+2k+4}}\right\}\left\{1+\frac{s_1(\varphi_{2i+2j+3})}{\alpha_{2i+2j+3}}\right\}}, \\
	&= \frac{\alpha_{2i+1}}{\varphi_{2i+3}}\frac{S_{2i+1}\sum_{j=0}^{n}\left(\prod_{k=0}^{j-1}\frac{\varphi_{2i+2k+2}\,\varphi_{2i+2k+3}}{\alpha_{2i+2k+1}\,\alpha_{2i+2k+2}}\right)\left(\frac{S_{2i+2j+2}}{S_{2i+2j+1}}+\frac{\varphi_{2i+2j+2}\,S_{2i+2j+2}}{\alpha_{2i+2j+1}\,\alpha_{2i+2j+2}\,S_{2i+2j+3}}\right)}{S_{2i+3}\sum_{j=0}^{n}\left(\prod_{k=0}^{j-1}\frac{\varphi_{2i+2k+4}\,\varphi_{2i+2k+5}}{\alpha_{2i+2k+3}\,\alpha_{2i+2k+4}}\right)\left(\frac{S_{2i+2j+4}}{S_{2i+2j+3}}+\frac{\varphi_{2i+2j+4}\,S_{2i+2j+4}}{\alpha_{2i+2j+3}\,\alpha_{2i+2j+4}\,S_{2i+2j+5}}\right)}.
\end{align*}
Then we have
\begin{align*}
	&\varphi_{2i+1}\sum_{j=0}^{n}\left(\prod_{k=0}^{j-1}\frac{\varphi_{2i+2k+2}\,\varphi_{2i+2k+3}}{\alpha_{2i+2k+1}\,\alpha_{2i+2k+2}}\right)\left(\frac{S_{2i+2j+2}}{S_{2i+2j+1}}+\frac{\varphi_{2i+2j+2}\,S_{2i+2j+2}}{\alpha_{2i+2j+1}\,\alpha_{2i+2j+2}\,S_{2i+2j+3}}\right) \\
	&= \varphi_{2i+1}\sum_{j=0}^{n}\left(\prod_{k=0}^{j-1}\frac{\varphi_{2i+2k+2}\,\varphi_{2i+2k+3}}{\alpha_{2i+2k+1}\,\alpha_{2i+2k+2}}\right)\frac{1}{\varphi_{2i+2j+1}}\left(1+\frac{\alpha_{2i+2j}}{\varphi_{2i+2j}}-\frac{S_{2i+2j}}{S_{2i+2j+1}}+\frac{\varphi_{2i+2j+1}}{\alpha_{2i+2j+1}}\frac{\varphi_{2i+2j+2}}{\alpha_{2i+2j+2}}\frac{S_{2i+2j+2}}{S_{2i+2j+3}}\right) \\
	&= \sum_{j=0}^{n}\left(\prod_{k=0}^{j-1}\frac{\varphi_{2i+2k+1}}{\alpha_{2i+2k+1}}\frac{\varphi_{2i+2k+2}}{\alpha_{2i+2k+2}}\right)\left(1+\frac{\alpha_{2i+2j}}{\varphi_{2i+2j}}\right) \\
	&\quad + \sum_{j=0}^{n}\left(\prod_{k=0}^{j-1}\frac{\varphi_{2i+2k+1}}{\alpha_{2i+2k+1}}\frac{\varphi_{2i+2k+2}}{\alpha_{2i+2k+2}}\right)\left(\frac{\varphi_{2i+2j+1}}{\alpha_{2i+2j+1}}\frac{\varphi_{2i+2j+2}}{\alpha_{2i+2j+2}}\frac{S_{2i+2j+2}}{S_{2i+2j+3}}-\frac{S_{2i+2j}}{S_{2i+2j+1}}\right),
\end{align*}
by using equation \eqref{Eq:Action_s'1s1_Lem_1}.
The first term of the right-hand side is equivalent to $\frac{\alpha_{2i}}{\varphi_{2i}}\,S'_{2i}$.
Moreover, we rewrite the second term as
\begin{align*}
	&\sum_{j=0}^{n}\left(\prod_{k=0}^{j-1}\frac{\varphi_{2i+2k+1}}{\alpha_{2i+2k+1}}\frac{\varphi_{2i+2k+2}}{\alpha_{2i+2k+2}}\right)\left(\frac{\varphi_{2i+2j+1}}{\alpha_{2i+2j+1}}\frac{\varphi_{2i+2j+2}}{\alpha_{2i+2j+2}}\frac{S_{2i+2j+2}}{S_{2i+2j+3}}-\frac{S_{2i+2j}}{S_{2i+2j+1}}\right) \\
	&= \left(\prod_{j=0}^{n}\frac{\varphi_{2i+2j+1}}{\alpha_{2i+2j+1}}\frac{\varphi_{2i+2j+2}}{\alpha_{2i+2j+2}}\right)\frac{S_{2i+2n+2}}{S_{2i+2n+3}} - \frac{S_{2i}}{S_{2i+1}} \\
	&= \left(\frac{\beta'_0}{q}-1\right)\frac{S_{2i}}{S_{2i+1}} \\
	&= \frac{\frac{\varphi_{2i}}{\alpha_{2i}}\,S_{2i}\,S'_{2i+1}-S_{2i}\,S'_{2i}}{S_{2i+1}},
\end{align*}
by using equation \eqref{Eq:Action_s'1s1_Lem_3}.
It follows that
\begin{equation}\label{Eq:Action_s'1s1_2}
	s'_1\,s_1(\varphi_{2i+1}) = \frac{\alpha_{2i+1}}{\varphi_{2i+1}}\frac{\frac{\alpha_{2i}}{\varphi_{2i}}\,S'_{2i}\,S_{2i+1}+\frac{\varphi_{2i}}{\alpha_{2i}}\,S_{2i}\,S'_{2i+1}-S_{2i}\,S'_{2i}}{\frac{\alpha_{2i+2}}{\varphi_{2i+2}}\,S'_{2i+2}\,S_{2i+3}+\frac{\varphi_{2i+2}}{\alpha_{2i+2}}\,S_{2i+2}\,S'_{2i+3}-S_{2i+2}\,S'_{2i+2}}.
\end{equation}
In a similar manner, we obtain
\[
	s_1\,s'_1(\varphi_{2i+1}) = \frac{\alpha_{2i+1}}{\varphi_{2i+1}}\frac{\frac{\alpha_{2i}}{\varphi_{2i}}\,S'_{2i}\,S_{2i+1}+\frac{\varphi_{2i}}{\alpha_{2i}}\,S_{2i}\,S'_{2i+1}-S_{2i}\,S'_{2i}}{\frac{\alpha_{2i+2}}{\varphi_{2i+2}}\,S'_{2i+2}\,S_{2i+3}+\frac{\varphi_{2i+2}}{\alpha_{2i+2}}\,S_{2i+2}\,S'_{2i+3}-S_{2i+2}\,S'_{2i+2}},
\]
by using equation \eqref{Eq:Action_s'1s1_Lem_2} and \eqref{Eq:Action_s'1s1_Lem_4}.

The proof has been just finished.
In addition, we give the following lemma, which will be used together with equation \eqref{Eq:Action_s'1s1_1} in Appendix \ref{App:Prf_q-FST} to derive system \eqref{Eq:q-FST}.

\begin{lem}
An equation
\begin{equation}\label{Eq:Action_s'1s1_3}
	s'_1\,s_1(\varphi_{2i+1}) = \frac{1+\frac{\alpha_{2i+2}}{\varphi_{2i+2}}}{1+\frac{\alpha_{2i}}{\varphi_{2i}}}\,\frac{\alpha_{2i+1}}{\varphi_{2i+1}}\,\frac{S_{2i}\,S'_{2i+2}+\frac{\alpha_{2i}}{\varphi_{2i}}\,\alpha_{2i+1}\,S'_{2i}\,S_{2i+2}}{S_{2i+2}\,S'_{2i+4}+\frac{\alpha_{2i+2}}{\varphi_{2i+2}}\,\alpha_{2i+3}\,S'_{2i+2}\,S_{2i+4}}\quad (i=0,\ldots,n),
\end{equation}
is satisfied.
\end{lem}

\begin{proof}
For each $i=0,\ldots,n$, equation \eqref{Eq:Action_s'1s1_2} is rewritten as
\begin{equation}\label{Eq:Action_s'1s1_Lem_5}
	s'_1\,s_1(\varphi_{2i+1}) = \frac{\alpha_{2i+1}}{\varphi_{2i+1}}\frac{\frac{\alpha_{2i}}{\varphi_{2i}}\,S'_{2i}\,S_{2i+1}+\left(\frac{\beta'_0}{q}-1\right)S_{2i}}{\frac{\alpha_{2i+2}}{\varphi_{2i+2}}\,S'_{2i+2}\,S_{2i+3}+\left(\frac{\beta'_0}{q}-1\right)S_{2i+2}}.
\end{equation}
On the other hand, in a similar manner as \eqref{Eq:Action_s'1s1_Lem_1}, we have
\begin{equation}\label{Eq:Action_s'1s1_Lem_6}
	\frac{\alpha_{2i}}{\varphi_{2i}}\,S'_{2i} - \frac{\varphi_{2i+1}}{\alpha_{2i+1}}\,S'_{2i+2} = \left(1-\frac{\beta'_0}{q}\right)\left(1+\frac{\alpha_{2i}}{\varphi_{2i}}\right).
\end{equation}
Equations \eqref{Eq:Action_s'1s1_Lem_1} and \eqref{Eq:Action_s'1s1_Lem_6} give
\begin{equation}\begin{split}\label{Eq:Action_s'1s1_Lem_7}
    	\left(1+\frac{\alpha_{2i}}{\varphi_{2i}}\right)\left\{\frac{\alpha_{2i}}{\varphi_{2i}}\,S'_{2i}\,S_{2i+1}+\left(\frac{\beta'_0}{q}-1\right)S_{2i}\right\}
	&= \frac{\alpha_{2i}}{\varphi_{2i}}\left(1+\frac{\alpha_{2i}}{\varphi_{2i}}\right)S'_{2i}\,S_{2i+1} + \frac{\varphi_{2i+1}}{\alpha_{2i+1}}\,S_{2i}\,S'_{2i+2} - \frac{\alpha_{2i}}{\varphi_{2i}}\,S_{2i}\,S'_{2i} \\
	&= \frac{\varphi_{2i+1}}{\alpha_{2i+1}}S_{2i}\,S'_{2i+2} + \frac{\alpha_{2i}}{\varphi_{2i}}\,\varphi_{2i+1}\,S'_{2i}\,S_{2i+2}.
\end{split}\end{equation}
Substituting \eqref{Eq:Action_s'1s1_Lem_7} to \eqref{Eq:Action_s'1s1_Lem_5}, we obtain equation \eqref{Eq:Action_s'1s1_3}.
\end{proof}

\section{Proof of Theorem \ref{Thm:q-FST}}\label{App:Prf_q-FST}

We prove the first half of the theorem.
The actions
\[
	\tau_1(a_i) = a_i,\quad
	\tau_1(b_i) = b_i\quad (i=1,\ldots,n+1),\quad
	\tau_1(t) = q\,t,
\]
are obtained immediately from
\begin{equation}\label{Proof:Gen_q-PVI_par}
	a_i^{2n+2} = \frac{q^{n-2i+3}\,\beta_0}{\alpha_{2i}\,\alpha_{2i+1}^2\ldots\alpha_{2i+2n}^{2n+1}\,\beta'_0},\quad
	b_i^{2n+2} = \frac{q^{n-2i+2}\,\beta_0}{\alpha_{2i+1}\,\alpha_{2i+2}^2\ldots\alpha_{2i+2n+1}^{2n+1}\,\beta'_0}\quad (i=1,\ldots,n+1),\quad
	t = \beta'_0.
\end{equation}
Hence the remaining problems are the actions $\tau_1(f_i)$ and $\tau_1(g_i)$.

The action $\tau_1(g_i)$ is derived as follows.
Equation \eqref{Eq:Action_s'1s1_1} implies
\begin{equation}\label{Proof:Gen_q-PVI_g_1}
	\tau_1(\varphi_{2i}) = \alpha_{2i}\,\varphi_{2i}\,\alpha_{2i+1}\,\frac{\hat{S}'_{2i}\,\hat{S}_{2i+2}}{\hat{S}_{2i}\,\hat{S}'_{2i+2}}\quad (i=1,\ldots,n),
\end{equation}
where
\begin{equation}\begin{split}\label{Proof:Gen_q-PVI_g_2}
	\hat{S}_{2i} &= \sum_{j=i}^{n}\left(\prod_{k=i}^{j-1}\varphi_{2k}\frac{\alpha_{2k+1}}{\varphi_{2k+1}}\right)(1+\varphi_{2j}) + \left(\prod_{k=i}^{n}\varphi_{2k}\frac{\alpha_{2k+1}}{\varphi_{2k+1}}\right)(1+\varphi_0) \\
	&\quad + \sum_{j=1}^{i-1}\left(\prod_{k=i}^{n}\varphi_{2k}\frac{\alpha_{2k+1}}{\varphi_{2k+1}}\right)\left(\prod_{k=0}^{j-1}\varphi_{2k}\frac{\alpha_{2k+1}}{\varphi_{2k+1}}\right)(1+\varphi_{2j}), \\
	\hat{S}'_{2i} &= \sum_{j=i}^{n}\left(\prod_{k=i}^{j-1}\frac{1}{\varphi_{2k}}\frac{1}{\varphi_{2k+1}}\right)\left(1+\frac{1}{\varphi_{2j}}\right) + \left(\prod_{k=i}^{n}\frac{1}{\varphi_{2k}}\frac{1}{\varphi_{2k+1}}\right)\left(1+\frac{1}{\varphi_0}\right) \\
	&\quad + \sum_{j=1}^{i-1}\left(\prod_{k=i}^{n}\frac{1}{\varphi_{2k}}\frac{1}{\varphi_{2k+1}}\right)\left(\prod_{k=0}^{j-1}\frac{1}{\varphi_{2k}}\frac{1}{\varphi_{2k+1}}\right)\left(1+\frac{1}{\varphi_{2j}}\right),
\end{split}\end{equation}
for $i=0,\ldots,n$.
We assume that the indices of $\hat{S}_i,\hat{S}'_i$ are congruent modulo $2n+2$ as well as $\alpha_i,\varphi_i,S_i,S'_i$ are.
On the other hand, the definition of $f_i,g_i$ implies
\begin{equation}\begin{split}\label{Proof:Gen_q-PVI_g_3}
	\varphi_{2i} &= -\frac{b_i}{g_i}\quad (i=0,\ldots,n), \\
	\varphi_1 &= -\frac{g_0\,(b_1-g_1)}{b_1\,(b_0-g_0)}\frac{t}{f_1}, \\
	\varphi_{2i+1} &= -\frac{g_i\,(b_{i+1}-g_{i+1})}{b_{i+1}\,(b_i-g_i)}\frac{f_i}{f_{i+1}}\quad (i=1,\ldots,n-1), \\
	\varphi_{2n+1} &= -\frac{g_n\,(b_0-g_0)}{b_0\,(b_n-g_n)}\,f_n.
\end{split}\end{equation}
Substituting \eqref{Proof:Gen_q-PVI_g_3} to \eqref{Proof:Gen_q-PVI_g_2}, we have
\begin{equation}\begin{split}\label{Proof:Gen_q-PVI_g_4}
	&\hat{S}_{2i} = -\frac{b_i-g_i}{q^{\frac{n-2}{2}}t\,g_0\,f_i}\left(\prod_{j=i}^{n}\frac{1}{g_j^2}\right)G_i,\quad
	\hat{S}'_{2i} = \frac{b_i-g_i}{t\,b_i\,f_i}\,F_i\quad(i=1,\ldots,n), \\
	&\hat{S}_{2n+2} = -\frac{b_0-g_0}{q^{\frac{n}{2}}t\,g_0}\,G_{n+1},\quad
	\hat{S}'_{2n+2} = \frac{b_0-g_0}{t\,b_0}\,F_{n+1}.
\end{split}\end{equation}
Substituting \eqref{Proof:Gen_q-PVI_g_3} and \eqref{Proof:Gen_q-PVI_g_4} to \eqref{Proof:Gen_q-PVI_g_1}, we obtain
\begin{equation}\label{Proof:Gen_q-PVI_g_5}
	\tau_1(g_i) = \frac{1}{g_i}\frac{G_i\,F_{i+1}}{F_i\,G_{i+1}}\quad (i=1,\ldots,n).
\end{equation}
Note that
\[
	\tau_1(g_0) = \frac{1}{q^{\frac n2}t\,\tau_1(g_1)\ldots\tau_1(g_n)} = \frac{1}{q^{n-1}t^2\,g_0}\frac{G_{n+1}\,F_1}{F_{n+1}\,G_1}.
\]

The action $\tau_1(f_i)$ is derived as follows.
Equation \eqref{Eq:Action_s'1s1_3} implies
\[
	\tau_1(\varphi_{2i-1}) = \frac{1+\varphi_{2i}}{1+\varphi_{2i-2}}\,\varphi_{2i+1}\,\frac{\hat{S}_{2i-2}\,\hat{S}'_{2i}+\varphi_{2i-2}\,\alpha_{2i-1}\,\hat{S}'_{2i-2}\,\hat{S}_{2i}}{\hat{S}_{2i}\,\hat{S}'_{2i+2}+\varphi_{2i}\,\alpha_{2i+1}\,\hat{S}'_{2i}\,\hat{S}_{2i+2}}\quad (i=1,\ldots,n).
\]
It follows that
\begin{equation}\label{Proof:Gen_q-PVI_f_1}
	\tau_1(\varphi_{2i-1}\,\varphi_{2i}) = \frac{1+\varphi_{2i}}{1+\varphi_{2i-2}}\,\alpha_{2i}\,\varphi_{2i}\,\alpha_{2i+1}\,\varphi_{2i+1}\,\frac{\hat{S}'_{2i}\,\hat{S}_{2i+2}}{\hat{S}_{2i}\,\hat{S}'_{2i+2}}\frac{\hat{S}_{2i-2}\,\hat{S}'_{2i}+\varphi_{2i-2}\,\alpha_{2i-1}\,\hat{S}'_{2i-2}\,\hat{S}_{2i}}{\hat{S}_{2i}\,\hat{S}'_{2i+2}+\varphi_{2i}\,\alpha_{2i+1}\,\hat{S}'_{2i}\,\hat{S}_{2i+2}}\quad (i=1,\ldots,n).
\end{equation}
Substituting \eqref{Proof:Gen_q-PVI_g_3} and \eqref{Proof:Gen_q-PVI_g_4} to \eqref{Proof:Gen_q-PVI_f_1}, we have
\begin{equation}\begin{split}\label{Proof:Gen_q-PVI_f_2}
	\tau_1\left(\frac{\frac{b_1}{g_1}-1}{\frac{b_0}{g_0}-1}\frac{t}{f_1}\right) &= \frac{f_1\,g_1\,F_1\,G_2}{q^{n-1}t^2\,g_0\,G_1\,F_2}\frac{G_{n+1}\,F_1-q^{n-1}t^2\,g_0\,F_{n+1}\,a_1\,G_1}{G_1\,F_2-g_1\,F_1\,a_2\,G_2}, \\
	\tau_1\left(\frac{\frac{b_{i+1}}{g_{i+1}}-1}{\frac{b_i}{g_i}-1}\frac{f_i}{f_{i+1}}\right) &= \frac{f_{i+1}\,g_{i+1}\,F_{i+1}\,G_{i+2}}{f_i\,g_i\,G_{i+1}\,F_{i+2}}\frac{G_i\,F_{i+1}-g_i\,F_i\,a_{i+1}\,G_{i+1}}{G_{i+1}\,F_{i+2}-g_{i+1}\,F_{i+1}\,a_{i+2}\,G_{i+2}}\quad (i=1,\ldots,n).
\end{split}\end{equation}
Combining \eqref{Proof:Gen_q-PVI_g_5} and \eqref{Proof:Gen_q-PVI_f_2}, we obtain
\begin{align*}
	\tau_1\left(\frac{f_1}{t}\right) &= \frac{1}{f_1}\,\frac{F_2}{F_{n+1}}\frac{\frac{b_1}{\tau_1(g_1)}-1}{\frac{b_0}{\tau_1(g_0)}-1}\frac{\tau_1(g_1)-a_2}{\tau_1(g_0)-a_1}, \\
	\tau_1\left(\frac{f_{i+1}}{f_i}\right) &= \frac{f_i}{f_{i+1}}\frac{F_{i+2}}{F_i}\frac{\frac{b_{i+1}}{\tau_1(g_{i+1})}-1}{\frac{b_i}{\tau_1(g_i)}-1}\frac{\tau_1(g_{i+1})-a_{i+2}}{\tau_1(g_i)-a_{i+1}}\quad (i=1,\ldots,n-1).
\end{align*}

We next prove the latter half of the theorem.
The actions of $r_0,\ldots,r_{2n+1},s'_1\,\pi$ on the parameters and the independent variable are obtained from \eqref{Proof:Gen_q-PVI_par}.
The actions of $r_0,\ldots,r_{2n+1}$ on the dependent variables are derived by direct calculations with
\[
	\frac{\alpha_{2i+1}+\varphi_{2i+1}}{1+\varphi_{2i+1}} = \frac{R_{i+1}^{b,a,a}}{R_{i+1}^{b,b,b}}\quad (i=0,\ldots,n),
\]
and
\begin{align*}
	R_{i+1}^{b,a,a} - \frac{a_{i+1}}{g_{i+1}}\,R_{i+1}^{b,b,b} &= \left(1-\frac{b_{i+1}}{g_{i+1}}\right)R_{i+1}^{a,a,a}\quad (i=0,\ldots,n-1), \\
	R_{n+1}^{b,a,a} - \frac{q\,a_{n+1}}{g_0}\,R_{n+1}^{b,b,b} &= \left(1-\frac{b_0}{g_0}\right)R_{n+1}^{a,a,a}, \\
	R_{i+1}^{b,b,b} - \frac{b_i}{g_i}\,R_{i+1}^{b,a,a} &= \left(1-\frac{b_i}{g_i}\right)R_{i+1}^{b,a,b}\quad (i=0,\ldots,n).
\end{align*}
The action of $s'_1\,\pi$ on the dependent variables is derived as follows.
We have
\[
	s'_1\,\pi(\varphi_{2i}) = \frac{\alpha_{2i+1}\,\alpha_{2i+2}}{\varphi_{2i+2}}\frac{S'_{2i+1}}{S'_{2i+3}},\quad
	s'_1\,\pi(\varphi_{2i+1}) = \frac{\alpha_{2i+2}\,\alpha_{2i+3}}{\varphi_{2i+3}}\frac{S'_{2i+2}}{S'_{2i+4}}\quad (i=0,\ldots,n).
\]
It follows that
\begin{equation}\begin{split}\label{Proof:Gen_q-PVI_Weyl_1}
	s'_1\,\pi\left(\frac{f_i}{f_{i+1}}\right) &= \frac{\alpha_{2i+2}\,\alpha_{2i+3}}{\varphi_{2i+3}}\frac{\alpha_{2i+3}\,\alpha_{2i+4}}{\varphi_{2i+4}}\frac{\frac{\alpha_{2i+1}\,\alpha_{2i+2}}{\varphi_{2i+2}}\,S'_{2i+1}+S'_{2i+3}}{\frac{\alpha_{2i+3}\,\alpha_{2i+4}}{\varphi_{2i+4}}\,S'_{2i+3}+S'_{2i+5}}\frac{S'_{2i+2}}{S'_{2i+4}}\quad (i=1,\ldots,n-1), \\
	s'_1\,\pi(f_n) &= \frac{\alpha_0\,\alpha_1}{\varphi_1}\frac{\alpha_1\,\alpha_2}{\varphi_2}\frac{\frac{\alpha_{2n+1}\,\alpha_0}{\varphi_0}\,S'_{2n+1}+S'_1}{\frac{\alpha_1\,\alpha_2}{\varphi_2}\,S'_1+S'_3}\frac{S'_0}{S'_2}, \\
	s'_1\,\pi\left(\frac{g_i}{b_i}\right) &= -\frac{\varphi_{2i+2}}{\alpha_{2i+1}\,\alpha_{2i+2}}\frac{S'_{2i+3}}{S'_{2i+1}}\quad (i=1,\ldots,n).
\end{split}\end{equation}
On the other hand, substituting \eqref{Proof:Gen_q-PVI_g_3} to
\[
	S'_{2i+1} = \sum_{j=0}^{n}\left(\prod_{k=0}^{j-1}\frac{\varphi_{2i+2k+1}}{\alpha_{2i+2k+1}}\frac{\varphi_{2i+2k+2}}{\alpha_{2i+2k+2}}\right)\left(1+\frac{\varphi_{2i+2j+1}}{\alpha_{2i+2j+1}}\right)\quad (i=0,\ldots,n),
\]
we have
\begin{equation}\label{Proof:Gen_q-PVI_Weyl_2}
	S'_{2i+1} = -\frac{g_i}{a_{i+1}}\,R_i^{*}\quad (i=0,\ldots,n).
\end{equation}
Equations \eqref{Eq:Action_s'1s1_Lem_2}, \eqref{Eq:Action_s'1s1_Lem_4} and \eqref{Proof:Gen_q-PVI_Weyl_2} imply
\begin{equation}\begin{split}\label{Proof:Gen_q-PVI_Weyl_3}
	S'_{2i+2} &= \frac{\alpha_{2i+1}}{\varphi_{2i+1}}\left(S'_{2i+1}-1+\frac tq\right) = R_{i+1}^{*}+1-\frac tq\quad (i=0,\ldots,n-1), \\
	S'_0 &= \frac{\alpha_{2n+1}}{\varphi_{2n+1}}\left(S'_{2n+1}-1+\frac tq\right) = R_0^{*}+1-\frac tq.
\end{split}\end{equation}
Substituting \eqref{Proof:Gen_q-PVI_g_3}, \eqref{Proof:Gen_q-PVI_Weyl_2} and \eqref{Proof:Gen_q-PVI_Weyl_3} to \eqref{Proof:Gen_q-PVI_Weyl_1}, we obtain
\begin{align*}
	s'_1\,\pi\left(\frac{f_i}{f_{i+1}}\right) &= \frac{(g_i\,R_i^{*}-b_{i+1}\,R_{i+1}^{*})(b_{i+1}-g_{i+1})\left(R_{i+1}^{*}+1-\frac tq\right)f_{i+2}}{(g_{i+1}\,R_{i+1}^{*}-b_{i+2}\,R_{i+2}^{*})(b_{i+2}-g_{i+2})\left(R_{i+2}^{*}+1-\frac tq\right)f_{i+1}}\quad (i=1,\ldots,n-2), \\
	s'_1\,\pi\left(\frac{f_{n-1}}{f_n}\right) &= q\,\frac{(g_{n-1}\,R_{n-1}^{*}-b_n\,R_n^{*})(b_n-g_n)\left(R_n^{*}+1-\frac tq\right)}{(g_n\,R_n^{*}-b_{n+1}\,R_0^{*})(b_0-g_0)\left(R_0^{*}+1-\frac tq\right)f_n}, \\
	s'_1\,\pi(f_n) &= q\,\frac{(g_n\,R_n^{*}-b_{n+1}\,R_0^{*})(b_0-g_0)\left(R_0^{*}+1-\frac tq\right)f_1}{(g_0\,R_0^{*}-b_1\,R_1^{*})(b_1-g_1)\left(R_1^{*}+1-\frac tq\right)t}, \\
	s'_1\,\pi\left(\frac{g_i}{b_i}\right) &= \frac{b_{i+1}\,R_{i+1}^{*}}{g_i\,R_i^{*}}\quad (i=1,\ldots,n-1), \\
	s'_1\,\pi\left(\frac{g_n}{b_n}\right) &= \frac{b_{n+1}\,R_0^{*}}{g_n\,R_n^{*}}.
\end{align*}

\section{Proof of Theorem \ref{Thm:q-Garnier}}\label{App:Prf_q-Garnier}

We prove the first half of the theorem.
It is obvious that the compatibility condition of \eqref{Eq:q-Garnier_Lax_1} is equivalent to that of \eqref{Eq:q-Garnier_Lax_2}.
Then the compatibility condition of \eqref{Eq:q-Garnier_Lax_2} is described as
\begin{equation}\begin{split}\label{Proof:q-Laplace_1}
	\tau_2(\hat{M}_{i,i})\,\Gamma_{i,i} &= \Gamma_{i,i}\,\hat{M}_{i,i}\quad (i=1,\ldots,n+1),
\end{split}\end{equation}
and
\begin{equation}\begin{split}\label{Proof:q-Laplace_2}
	\tau_2(\hat{M}_{i,i})\,\Gamma_{i,i+1} + \tau_2(\hat{M}_{i,i+1})\,\Gamma_{i+1,i+1} &= \Gamma_{i,i}\,\hat{M}_{i,i+1} + \Gamma_{i,i+1}\,\hat{M}_{i+1,i+1}\quad (i=1,\ldots,n), \\
	\tau_2(\hat{M}_{n+1,n+1})\,\Gamma_{n+1,1} + \tau_2(\hat{M}_{n+1,1})\,\Gamma_{1,1} &= \Gamma_{n+1,n+1}\,\hat{M}_{n+1,1} + q^{-1}\,\Gamma_{n+1,1}\,\hat{M}_{1,1}, \\
	\tau_2(\hat{M}_{i,i+1})\,\Gamma_{i+1,i+2} &= \Gamma_{i,i+1}\,\hat{M}_{i+1,i+2}\quad (i=1,\ldots,n-1), \\
	\tau_2(\hat{M}_{n,n+1})\,\Gamma_{n+1,1} &= \Gamma_{n,n+1}\,\hat{M}_{n+1,1}, \\
	\tau_2(\hat{M}_{n+1,1})\,\Gamma_{1,2} &= q^{-1}\,\Gamma_{n+1,1}\,\hat{M}_{1,2}.
\end{split}\end{equation}
Equation \eqref{Proof:q-Laplace_2} implies
\begin{equation}\begin{split}\label{Proof:q-Laplace_3}
	&\tau_2(\hat{M}_{i,i+1}^{-1})\left\{\Gamma_{i,i}+\left(z\,I-\tau_2(\hat{M}_{i,i})\right)\,\Gamma_{i,i+1}\,\hat{M}_{i,i+1}^{-1}\right\} \\
	&= \left\{\Gamma_{i+1,i+1}+\Gamma_{i+1,i+2}\,\hat{M}_{i+1,i+2}^{-1}\left(z\,I-\hat{M}_{i+1,i+1}\right)\right\}\hat{M}_{i,i+1}^{-1}\quad (i=1,\ldots,n-1), \\
	&\tau_2(\hat{M}_{n,n+1}^{-1})\left\{\Gamma_{n,n}+\left(z\,I-\tau_2(\hat{M}_{n,n})\right)\,\Gamma_{n,n+1}\,\hat{M}_{n,n+1}^{-1}\right\} \\
	&= \left\{\Gamma_{n+1,n+1}+\Gamma_{n+1,1}\,\hat{M}_{n+1,1}^{-1}\left(z\,I-\hat{M}_{n+1,n+1}\right)\right\}\hat{M}_{n,n+1}^{-1}, \\
	&\tau_2(\hat{M}_{n+1,1}^{-1})\left\{\Gamma_{n+1,n+1}+\left(z\,I-\tau_2(\hat{M}_{n+1,n+1})\right)\,\Gamma_{n+1,1}\,\hat{M}_{n+1,1}^{-1}\right\} \\
	&= \left\{\Gamma_{1,1}+q\,\Gamma_{1,2}\,\hat{M}_{1,2}^{-1}\left(z\,I-q^{-1}\,\hat{M}_{1,1}\right)\right\}\hat{M}_{n+1,1}^{-1}.
\end{split}\end{equation}
Then we obtain
\begin{align*}
	&\tau_2(\mathcal{A})\,\mathcal{B} \\
	&= \tau_2(\hat{M}_{n+1,1}^{-1})\ldots\left(z\,I-\tau_2(\hat{M}_{2,2})\right)\tau_2(\hat{M}_{1,2}^{-1})\left(z\,I-\tau_2(\hat{M}_{1,1})\right)\left\{\Gamma_{1,1}+\Gamma_{1,2}\,\hat{M}_{1,2}^{-1}\left(z\,I-\hat{M}_{1,1}\right)\right\} \\
	&= \tau_2(\hat{M}_{n+1,1}^{-1})\ldots\left(z\,I-\tau_2(\hat{M}_{2,2})\right)\tau_2(\hat{M}_{1,2}^{-1})\left\{\Gamma_{1,1}+\left(z\,I-\tau_2(\hat{M}_{1,1})\right)\Gamma_{1,2}\,\hat{M}_{1,2}^{-1}\right\}\left(z\,I-\hat{M}_{1,1}\right) \\
	&= \tau_2(\hat{M}_{n+1,1}^{-1})\ldots \tau_2(\hat{M}_{2,3}^{-1})\left(z\,I-\tau_2(\hat{M}_{2,2})\right)\left\{\Gamma_{2,2}+\Gamma_{2,3}\,\hat{M}_{2,3}^{-1}\left(z\,I-\hat{M}_{2,2}\right)\right\}\hat{M}_{1,2}^{-1}\left(z\,I-\hat{M}_{1,1}\right) \\
	&= \ldots \\
	&= \tau_2(\hat{M}_{n+1,1}^{-1})\left(z\,I-\tau_2(\hat{M}_{n+1,n+1})\right)\left\{\Gamma_{n+1,n+1}+\Gamma_{n+1,1}\,\hat{M}_{n+1,1}^{-1}\left(z\,I-\hat{M}_{n+1,n+1}\right)\right\}\hat{M}_{n,n+1}^{-1}\ldots\left(z\,I-\hat{M}_{1,1}\right) \\
	&= \tau_2(\hat{M}_{n+1,1}^{-1})\left\{\Gamma_{n+1,n+1}+\left(z\,I-\tau_2(\hat{M}_{n+1,n+1})\right)\Gamma_{n+1,1}\,\hat{M}_{n+1,1}^{-1}\right\}\left(z\,I-\hat{M}_{n+1,n+1}\right)\hat{M}_{n,n+1}^{-1}\ldots\left(z\,I-\hat{M}_{1,1}\right) \\
	&= \left\{\Gamma_{1,1}+\Gamma_{1,2}\,\hat{M}_{1,2}^{-1}\left(q\,z\,I-\hat{M}_{1,1}\right)\right\}\hat{M}_{n+1,1}^{-1}\left(z\,I-\hat{M}_{n+1,n+1}\right)\ldots\left(z\,I-\hat{M}_{1,1}\right) \\
	&= T_{q,z}(\mathcal{B})\,\mathcal{A},
\end{align*}
by using \eqref{Proof:q-Laplace_1} and \eqref{Proof:q-Laplace_3}.

We next prove the latter half of the theorem.

(1) It is obvious that the matrix $\mathcal{A}(z)$ is a polynomial of $(n+1)$-st order in $z$.

(2) It is easy to verify that
\[
	\mathcal{A}_{n+1} = \hat{M}_{n+1,1}^{-1}\,\hat{M}_{n,n+1}^{-1}\ldots\hat{M}_{1,2}^{-1} = a_1^{n+1}\begin{pmatrix}-t^{-1}&0\\{*}&-1\end{pmatrix}\begin{pmatrix}-1&0\\{*}&-1\end{pmatrix}\ldots\begin{pmatrix}-1&0\\{*}&-1\end{pmatrix} = (-a_1)^{n+1}\begin{pmatrix}t^{-1}&0\\{*}&1\end{pmatrix}.
\]
Hence the remaining problem is the eigenvalues of the matrix $\mathcal{A}_0$.
We obtain
\begin{align*}
	\det M
	&= a_1^{2n+2}\,\det\begin{pmatrix}\hat{M}_{1,1}&\hat{M}_{1,2}&&&&\\&\hat{M}_{2,2}&\hat{M}_{2,3}&&&\\&&\hat{M}_{3,3}&&\\&&&\ddots&&\\&&&&\hat{M}_{n,n}&\hat{M}_{n,n+1}\\z\,\hat{M}_{n+1,1}&&&&&\hat{M}_{n+1,n+1}\end{pmatrix} \\
	&= t\,\det\begin{pmatrix}\hat{M}_{1,2}^{-1}\,\hat{M}_{1,1}&I&&&&\\&\hat{M}_{2,3}^{-1}\,\hat{M}_{2,2}&I&&&\\&&\hat{M}_{3,4}^{-1}\,\hat{M}_{3,3}&&\\&&&\ddots&&\\&&&&\hat{M}_{n,n+1}^{-1}\,\hat{M}_{n,n}&I\\z\,I&&&&&\hat{M}_{n+1,1}^{-1}\,\hat{M}_{n+1,n+1}\end{pmatrix} \\
	&= t\,\det\begin{pmatrix}O&I&&&&\\-\hat{M}_{2,3}^{-1}\,\hat{M}_{2,2}\,\hat{M}_{1,2}^{-1}\,\hat{M}_{1,1}&{*}&I&&&\\&&\hat{M}_{3,4}^{-1}\,\hat{M}_{3,3}&&\\&&&\ddots&&\\&&&&\hat{M}_{n,n+1}^{-1}\,\hat{M}_{n,n}&I\\z\,I&&&&&\hat{M}_{n+1,1}^{-1}\,\hat{M}_{n+1,n+1}\end{pmatrix} \\
	&= t\,\det\begin{pmatrix}O&I&&&&\\O&{*}&I&&&\\\hat{M}_{3,4}^{-1}\,\hat{M}_{3,3}\,\hat{M}_{2,3}^{-1}\,\hat{M}_{2,2}\,\hat{M}_{1,2}^{-1}\,\hat{M}_{1,1}&&{*}&&\\&&&\ddots&&\\&&&&{*}&I\\z\,I&&&&&\hat{M}_{n+1,1}^{-1}\,\hat{M}_{n+1,n+1}\end{pmatrix} \\
	&= \ldots \\
	&= t\,\det\begin{pmatrix}O&I&&&&\\O&{*}&I&&&\\O&&{*}&&\\&&&\ddots&&\\(-1)^{n-1}\,\hat{M}_{n,n+1}^{-1}\,\hat{M}_{n,n}\ldots\hat{M}_{1,2}^{-1}\,\hat{M}_{1,1}&&&&{*}&I\\z\,I&&&&&\hat{M}_{n+1,1}^{-1}\,\hat{M}_{n+1,n+1}\end{pmatrix} \\
	&= t\,\det\begin{pmatrix}O&I&&&&\\O&{*}&I&&&\\O&&{*}&&\\&&&\ddots&&\\O&&&&{*}&I\\z\,I-\mathcal{A}_0&&&&&{*}\end{pmatrix} \\
	&= t\,\det\left(z\,I-\mathcal{A}_0\right),
\end{align*}
by using
\begin{equation}\label{Proof:q-Garnier_1}
	\det\hat{M}_{i,i+1} = \frac{1}{a_1^2}\quad (i=1,\ldots,n),\quad
	\det\hat{M}_{n+1,1}^{-1} = \frac{t}{a_1^2}.
\end{equation}
On the other hand, thanks to Lemma 2.4 of \cite{FS2}, we have
\[
	\det M = \left(t\,z-q^{-\frac n2}\right)\left(z-q^{\frac n2}\,a_1\,b_1\ldots a_{n+1}\,b_{n+1}\right).
\]
It follows that the eigenvalues of the matrix $\mathcal{A}_0$ are $q^{-\frac n2}\,t^{-1}$ and $q^{\frac n2}\,a_1\,b_1\ldots a_{n+1}\,b_{n+1}$.

(3) We can show by using \eqref{Proof:q-Garnier_1} and
\[
	\det\left(z\,I-\hat{M}_{i,i}\right) = \det\begin{pmatrix}z-\frac{a_i}{a_1}&{*}\\0&z-\frac{b_i}{a_1}\end{pmatrix} = \left(z-\frac{a_i}{a_1}\right)\left(z-\frac{b_i}{a_1}\right)\quad (i=1,\ldots,n+1).
\]

(4) We show for the case $n=2m-1$.
The matrix $\mathcal{B}$ is rewritten as
\[
	\mathcal{B} = z\left(\Gamma_{1,2}\,\hat{M}_{1,2}^{-1}\right) + \left(\Gamma_{1,1}-\Gamma_{1,2}\,\hat{M}_{1,2}^{-1}\,\hat{M}_{1,1}\right),
\]
where
\[
	\Gamma_{1,2}\,\hat{M}_{1,2}^{-1} = a_1\begin{pmatrix}1&0\\\gamma_{2,3}&t^{\frac{1}{n+1}}\end{pmatrix}\begin{pmatrix}-1&0\\{*}&-1\end{pmatrix} = -a_1\begin{pmatrix}1&0\\{*}&t^{\frac{1}{n+1}}\end{pmatrix}.
\]
Then we can describe the determinant of the matrix $\mathcal{B}$ as
\[
	\det\mathcal{B} = t^{\frac{1}{n+1}}\,a_1^2\left(z-\zeta_1\right)\left(z-\zeta_2\right).
\]
The compatibility condition of \eqref{Eq:q-Garnier_Lax_4} implies
\[
	\det\tau_2(\mathcal{A})\,\det\mathcal{B} = \det T_{q,z}(\mathcal{B})\,\det\mathcal{A},
\]
which is described as
\[
	q^2\left(z-\frac{b_m}{q\,a_1}\right)\left(z-\frac{b_{n+1}}{q\,a_1}\right)\left(z-\zeta_1\right)\left(z-\zeta_2\right) = \left(q\,z-\zeta_1\right)\left(q\,z-\zeta_2\right)\left(z-\frac{b_m}{a_1}\right)\left(z-\frac{b_{n+1}}{a_1}\right).
\]
Then we obtain
\[
	\zeta_1 = \frac{b_m}{a_1},\quad
	\zeta_2 = \frac{b_{n+1}}{a_1}.
\]
We can show for the case $n=2m$ in a similar manner.

\section*{Acknowledgement}

The authors would like to express his gratitude to Professors Rei Inoue, Tetsu Masuda, Teruhisa Tsuda and Yasuhiko Yamada for helpful comments and advices.
This work was supported by JSPS KAKENHI Grant Number 15K04911.


\end{document}